\documentclass[12pt]{amsart}

\usepackage{amssymb}
\usepackage{fullpage}
\usepackage{graphicx}
\usepackage{hyperref}

\usepackage{lmodern}

\newtheorem{theorem}{Theorem}
\newtheorem{lemma}{Lemma}[section]
\newtheorem{proposition}[lemma]{Proposition}

\newtheorem{corollary}[lemma]{Corollary}

\newtheorem*{corollary*}{Corollary}

\theoremstyle{definition}
\newtheorem{definition}[lemma]{Definition}

\theoremstyle{remark}
\newtheorem{remark}[lemma]{Remark}

\newcommand\gp[1]{\langle #1\rangle}
\newcommand\presn[2]{\langle #1\mid #2\rangle}
\newcommand\set[1]{\{ #1 \}}
\newcommand\setof[2]{\{#1 \mid #2\}}

\DeclareMathOperator{\Var}{Var}
\DeclareMathOperator{\Aut}{Aut}
\DeclareMathOperator{\Autf}{Aut^f\hspace{-2pt}}
\DeclareMathOperator{\Stab}{Stab}
\DeclareMathOperator{\Sp}{Sp}
\DeclareMathOperator{\St}{St}
\DeclareMathOperator{\SL}{SL}
\DeclareMathOperator{\GL}{GL}

\DeclareMathOperator{\Mod}{Mod}

\newcommand\Z{{\mathbb Z}}

\let\al\alpha
\let\ga\gamma
\let\de\delta
\let\ep\varepsilon
\let\la\lambda
\let\th\theta
\let\ze\zeta

\let\si\sigma
\let\om\omega
\let\Ga\Gamma

\let\Th\Theta
\let\La\Lambda

\let\eset\emptyset
\let\seq\subseteq

\def\theenumi{(\roman{enumi})}  

\def\CB{{\mathcal B}}
\def\CE{{\mathcal E}}
\def\CF{{\mathcal F}}

\def\CL{{\mathcal L}}
\def\CP{{\mathcal P}}
\def\CQ{{\mathcal Q}}

\def\CS{{\mathcal S}}
\def\CT{{\mathcal T}}
\def\CU{{\mathcal U}}
\def\CV{{\mathcal V}}
\def\MN{{\mathbb{N}}}
\def\MZ{{\mathbb{Z}}}

\let\oldmarginpar\marginpar
\renewcommand\marginpar[1]{\-\oldmarginpar[\raggedleft\footnotesize #1]%
{\raggedright\footnotesize #1}}

\newcommand{\rb}[1]{{\left( #1 \right)}}

\catcode`@ = 11
\def\p@enumi{\thelemma}

\let\@savedlabel\label
\def\label#1{\@savedlabel{#1}\ifnum\@listdepth=1%
\protected@edef\@currentlabel{\theenumi}\@savedlabel{#1-it@m}\fi}

\title{Quadratic equations in the Grigorchuk group}

\author{Igor Lysenok}
\address{Steklov Institute of Mathematics, Gubkina str.\ 8, 119991
Moscow, Russia} \email{igor.lysenok@gmail.com}
\thanks{The first author has been partially supported by the Russian Foundation for Basic Research}

\author{Alexei Miasnikov}
\address{Department of Mathematics, Stevens Institute of Technology,
Hoboken, NJ, 07030 USA} 
\email{amiasnikov@gmail.com}

\author{Alexander Ushakov}
\address{Department of Mathematics, Stevens Institute of Technology,
Hoboken, NJ, 07030 USA} \email{sasha.ushakov@gmail.com}
\thanks{The third author has been partially supported by NSF grant DMS--0914773}

\begin{document}

\begin{abstract}
We provide an algorithm which, for a given quadratic equation in the Grigorchuk group
determines if it has a solution. As a corollary to our approach, we prove that
the group has a finite commutator width.


\end{abstract}

\maketitle


\section{Introduction}

The problem to determine if a given system of equations in an algebraic system ~$S$ has a solution
(the {\em Diophantine problem for $S$}) is hard for most algebraic systems. The reason is
that the problem is quite general and many natural specific decision problems for $S$ can be reduced to the
Diophantine problem.
For example, the word and the conjugacy problems for a group ~$G$ are very special cases of solving
equations in $G$. This generality is a natural source of motivation for studying the problem.
Furthermore, equations in $S$ can be viewed as a narrow fragment of the elementary theory of $S$.
In many cases, solving the Diophantine problem and providing a structural description of
solution sets of systems of equations is the first important step towards proving
the solvability of the whole elementary theory.
In particular, this is the case for the famous Tarski problem on the solvability of
the elementary theory of a non-abelian free group, see \cite{Kharlampovich_Myasnikov:2006}.
The positive solution of the Diophantine problem for free groups \cite{Makanin:1982} and
a deep study of properties of solution sets of systems of equations in free groups initiated in
\cite{Razborov:1987}
are at the very foundation of the known approach to the problem.

These two natural questions can be applied to any countable group $G$:
solve the Diophantine problem for $G$ and find a good structural description of solutions sets
of systems of equations in $G$.

Among the whole class of equations in a group, a subclass of {\em quadratic equations}
plays a special role.  By definition, these are equations in which every variable occurs
exactly twice. Under this restriction, equations in groups are much more treatable than in the general case,
compare for example \cite{Comerford_Edmunds:1981} and \cite{Makanin:1982}.
A reason is that natural equation transformations applied to quadratic equations
do not increase their complexity. This is related to the fact that quadratic equations in groups
have a nice geometric interpretation in terms of compact surfaces
(this may be attributed to folklore; see also \cite{Schupp:1980} or \cite{Lysenok_Myasnikov:2011}).
Although being quadratic is a rather restrictive property, it is still a wide class;
for example, the word and the conjugacy problems in a group are still special cases of quadratic equations.
It is worthwhile to mention that in many cases,
the class of quadratic equations is one of several types of ``building blocks'' for equations of a general form,
see \cite{Kharlampovich_Myasnikov:1998(2)}.

There are two classes of infinite groups where equations are well understood. The first
is finitely generated abelian groups. In this case, systems of equations are just
linear Diophantine systems over $\Z$.
The second is non-abelian free groups.
Equations in this case are more complicated but has been extensively studied.
Although there are many other classes of infinite groups where some reasonably
general results on equations are known,
at present they can be informally classified into two types: groups with a ``free-like'' behavior (e.g.\ Gromov
hyperbolic groups) or groups with ``abelian-like'' behavior (e.g.\ nilpotent groups).
(A number of deep results is known also for groups of ``mixed type''; see the monograph
~\cite{Casals-Ruiz_Kazachkov:2011}
for equations in free partially commutative groups.)

In this paper, we make an attempt to study equations in groups which belong to neither of these two types.
Namely, we take the known 3-generated Grigorchuk 2-group \cite{Grigorchuk:1980} of intermediate growth
and prove that the Diophantine problem for this
group in the special case of quadratic equations is solvable.

\begin{theorem} \label{thm:solvability}
There is an algorithm which for a given quadratic equation in the Grigorchuk group $\Ga$,
determines if it has a solution or not.
\end{theorem}

A notable feature of the Grigorchuk group $\Ga$ is its self-similarity in the sense
that $\Ga$ is commensurable with its nontrivial direct power.
More precisely, there is a
``splitting'' homomorphism ~$\psi$ of a subgroup $St_\Ga(1)$ of $\Ga$ of index 2 to
the direct product $\Ga\times\Ga$ of two copies of $\Ga$ such that the image of $\psi$
has index 8 in $\Ga\times\Ga$ (see \cite[Chapter VIII, Theorem ~28]{Harpe}).
There are two important properties of $\psi$ which give rise to a number of remarkable facts
about ~$\Ga$. The first property is that
each component $\psi_i: St_\Ga(1) \to \Ga$ of $\psi=(\psi_0,\psi_1)$
is a contracting map with respect to the word length on $\Ga$ defined for a canonical set of generators for $\Ga$.
This provides an effective solution of the word problem for $\Ga$ and is a key assertion
in the proof that $\Ga$ is a 2-group. The second property is a stronger version of the first one:
the splitting homomorphism $\psi$ itself is a contracting map with respect to a certain
length function defined on $\Ga$. A corollary is that the growth function of $\Ga$
is neither polynomial nor exponential.

Our proof of Theorem \ref{thm:solvability} is based essentially on the stronger version
of the contracting property of the splitting homomorphism $\psi$.
We use also the fact that
$\Ga$ is a torsion group though we think that this is not essential.
We hope that the theorem could be generalized to a wider class of groups of a self-similar nature
(though, of course, much technical work for this generalization has to be done).

Our main technical tool is defining a special splitting map $\Psi$ on equations in $\Ga$ which simulates
application of
the homomorphism $\psi$ when arbitrary values of variables are substituted into the equation.
It is not hard to see that for a quadratic equation, application of $\Psi$ produces two equations
which are also quadratic.
Because $\psi$ is contracting, the coefficients of new equations are shorter
than the coefficients of the original one. Although the complexity of the non-coefficient part of
the equation may increase, this is sufficient to apply an induction.

We apply our technique to prove another non-trivial property
of $\Ga$:

\begin{theorem} \label{thm:bounded-width}
There is a number $N$ such that any element of $\Ga$ belonging to the
commutator subgroup $[\Ga,\Ga]$ is a product of at most $N$ commutators in $\Ga$.
\end{theorem}

It is well-known that two quadratic words $x^2 y^2 z^2$ and $x^2 [y, z]$ are equivalent
up to a substitution of variables induced by an automorphism of the free group $F(x,y,z)$.
This implies equivalence $x_1^2 x_2^2 \dots x_{2n+1}^2 \sim x_1^2 [x_2,x_3] \dots [x_{2n},x_{2n+1}]$
and we have the following immediate consequence.

\begin{corollary*}
There is a number $N$ such that any element of $\Ga$ belonging to the
verbal subgroup generated by squares is a product of at most $N$ squares in $\Ga$.
\end{corollary*}

Note that we do not provide a bound on $N$ in Theorem \ref{thm:bounded-width}.

\section{The Grigorchuk group}

Let $\CT$ be an infinite rooted regular binary tree.
By definition, the vertex set of $\CT$ is the set $\{0,1\}^\ast$ of all finite binary words with
the empty word  $\varepsilon$ at the root.
Two words $u$ and ~$v$ are connected by an edge in $\CT$ if and only if one of them is obtained
from the other by adding one letter $x \in \{0,1\}$ at the end.
The tree $\CT$ is shown in Figure \ref{fi:bin_tree}.
\begin{figure}[h]
\centerline{\includegraphics[scale=0.6]{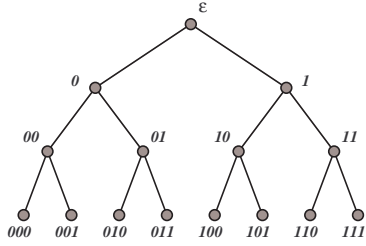}}
\caption{The infinite rooted regular binary tree $\CT$.}
\label{fi:bin_tree}
\end{figure}

By $\Aut(\CT)$ we denote the group of automorphisms of $\CT$.
Any automorphism $\al \in \Aut(\CT)$ can be viewed as a permutation on the set $\set{0,1}^*$
which preserves the length and initial segments, i.e.,\ $|\al(x)| = |x|$ for all $x$
and if $\al(xy) = x'y'$ and $|x|=|x'|$ then $\al(x) = x'$.
In particular, for every $n \ge 0$, $\al$ induces a permutation
on the set $\set{0,1}^n$ of words of length ~$n$ (the {\em $n$-th level} of $\CT$).
We denote by $\St(n)$ the stabilizer in $\Aut(\CT)$ of the set $\{0,1\}^n$.
In particular,
$$
  \St(1) = \setof{\al\in\Aut(\CT)}{ \al(0) = 0 \text{ and } \al(1) = 1}
$$
is the subgroup of $\Aut(\CT)$ of index 2.

Let $\CT_0$ and $\CT_1$ be the subtrees of $\CT$ spanned by the vertices starting with $0$ and $1$,
respectively.
By $a$ we denote the automorphism of $\CT$ which swaps $\CT_0$ and $\CT_1$:
$$
  \al(x w) = \bar x w \quad\text{for}\quad x \in \set{0,1}
$$
where $\bar x$ denotes $1-x$.

By definition, the Grigorchuk group $\Gamma$ is the  subgroup of $Aut(\CT)$ generated by four automorphisms
$a,b,c$ and $d$, where $b,c, d \in \St(1)$ are defined recursively as follows:
\begin{align*}
  b(0w) &= 0 a(w), &b(1w) &= 1c(w), \\
  c(0w) &= 0 a(w), &c(1w) &= 1d(w), \\
  d(0w) &= 0 w, &d(1w) &= 1b(w).
\end{align*}
It is easy to see that the generators $a,b,c$ and $d$ satisfy the relations
\begin{equation}\label{le:grig_relations}
    a^2=b^2=c^2=d^2=bcd=1.
\end{equation}
In particular,
$$
  \gp{a} = \set{1,a} \simeq \MZ / 2 \MZ \quad\mbox{and}\quad
  \langle b,c,d \rangle = \set{1,b,c,d} \simeq \MZ / 2 \MZ \times \MZ / 2 \MZ.
$$
Hence every element of $\Ga$ can be represented by a word of the form
\begin{equation}\label{eq:reduced-form}
[a]x_1ax_2a\ldots a x_n [a]
\end{equation}
where $x_i \in \{b,c,d\}$ and the first and the last occurrences of $a$ are optional.

Every automorphism $g \in \St(1)$ induces automorphisms $g_0$ and $g_1$ on the subtrees $\CT_0$ and ~$\CT_1$ of
$\CT$. Since $\CT_0$ and $\CT_1$ are naturally isomorphic to $\CT$ the mapping $g \mapsto (g_0,g_1)$ gives a group
isomorphism
    $$\psi: \St(1) \rightarrow \Aut(\CT) \times \Aut(\CT).$$
We denote by $\psi_i$ $(i=0,1)$ the components of $\psi$:
$$
  \psi(g) = (\psi_0(g), \ \psi_1(g)).
$$
Observe that conjugation by $a$ swaps the components of $\psi(g)$:
$$
  \psi(aga) = (\psi_1(g), \ \psi_0(g)).
$$

Let $St_\Gamma(1) = \St(1) \cap \Gamma$  be the set of automorphisms in $\Gamma$ stabilizing the
first level of ~$\CT$, i.e., stabilizing the vertices 0 and 1.
Since $b,c,d \in \St(1)$ and $a$ swaps $\CT_0$ and $\CT_1$, the subgroup
$St_\Ga(1)$ has index 2 in $\Ga$ and
a word $w$ represents an element of $St_\Gamma(1)$ if and only if
$w$ has an even number of occurrences of $a^{\pm1}$. This implies that $St_\Ga(1)$ has a generating set
$\set{b,c,d,aba,aca,ada}$.
From the definition of $b$, $c$ and $d$ we can write immediately the images under $\psi$
of the generators of $St_\Ga(1)$:
\begin{align*}
\psi(b) &= (a,c), & \psi(aba) &= (c,a), \\
\psi(c) &= (a,d), & \psi(aca) &= (d,a), \\
\psi(d) &= (1,b), & \psi(ada) &= (b,1).
\end{align*}

The monomorphism
    $$\psi: St_\Gamma(1) \rightarrow \Gamma \times \Gamma$$
plays a central role in our analysis of equations in $\Ga$.
Note that computation of $\psi$ is effective (for example, we can represent
an element of $St_\Ga(1)$ by a reduced word \eqref{eq:reduced-form}
as a concatenation of generators $\{b,c,d,aba,aca,ada\}$ and
then apply the formulas above).
%
%

We will need a description of the image of $\psi$ as well as an extra technical tool,
the ``subgroup $K$ trick'' (Proposition \ref{pr:psi-via-K}) used in \cite{Rozhkov:1998}
for a solution of the conjugacy problem for $\Ga$ (see also \cite{Lysenok_Miasnikov_Ushakov:2010}).
Let $K$ be the normal closure in $\Gamma$ of the element $abab$,
    $$K = \gp{abab}^\Gamma.$$


\begin{lemma}\label{le:K_split_component}
The following holds:
\begin{enumerate}
\item
$K$ has index $16$ in $\Gamma$  and
the quotient group $\Ga/K$ has the presentation
$$
    \Ga/K = \presn{a,b,d}{b^2 = a^2 = d^2 = 1, \ (ab)^2 = (bd)^2 = (ad)^4 = 1}
$$
\item
$\Ga/K$ is the direct product of the cyclic group of order 2 generated by $bK$
and the dihedral group of order 8 generated by $aK$ and $dK$.
\item
$K \times K \seq \psi(K)$.
\end{enumerate}
\end{lemma}

\begin{proof}
(ii) follows from (i). (iii) is Proposition 30(v) in \cite[Chapter VIII]{Harpe}.
Proposition 30(ii) in \cite[Chapter VIII]{Harpe} says that $K$ is of index 16.
To verify the presentation for $\Ga/K$ in (i) we first check that all defining relations hold in $\Ga/K$
and then compute that the presented group is of order 16.
%
%
\end{proof}

By $\pi_K$ we denote the natural epimorphism $\Ga \to \Ga/K$.
A straightforward consequence of Lemma ~\ref{le:K_split_component}(iii) is the following proposition.

\begin{proposition} \label{pr:psi-via-K}
There is a finite set ${\mathcal F}$ of pairs $(u,v) \in \Ga/K \times \Ga/K$ and a map
$\om: {\mathcal F} \to \Ga/K$ such that:
\begin{enumerate}
\item
A pair $(g_0,g_1) \in \Ga \times \Ga$ belongs to the image of $\psi$ if and only
$(\pi_K(g_0), \pi_K(g_1)) \in {\mathcal F}$.
\item
If $(\pi_K(g_0), \pi_K(g_1)) \in {\mathcal F}$ then for any $g \in \Ga$ with
$\psi(g) = (g_0,g_1)$, we have
$$
  \pi_K(g) = \om(\pi_K(g_0), \pi_K(g_1)).
$$
\end{enumerate}
\end{proposition}

\section{Quadratic equations} \label{se:quadratic}

\subsection{Equations in groups}
Let $G$ be a group and $X$ a countable set of {\em variables}.
An {\em equation} in $G$ is a formal equality $W=1$ where $W$
is a word $u_1 u_2 \dots u_k$ of letters $u_i \in G \cup X^{\pm1}$.
We view the left-hand side $W$ of an equation as an element of the free product $G*F_X$.
A {\em solution} of $W=1$ is a homomorphism $\al: G*F_X \to G$
which is identical on $G$ (i.e.\ $\al$ is a {\em $G$-homomorphism}) and satisfies $\al(W)=1$.
Similarly, a solution of a {\em system of equations} $\set{W_i=1}_{i \in I}$
is a $G$-homomorphism $\al: G*F_X \to G$ such that $\al(W_i)=1$ for all $i$.

For the Diophantine problem in a group $G$,
it is usually assumed that $G$ is finitely or countably
generated; in this case equations in $G$ can be represented by words in a countable alphabet
$A^{\pm1} \cup X^{\pm1}$ where $A$ is a generating set for $G$.

A word $W\in G \ast F_{X}$ and an equation $W=1$ are called {\em quadratic}
if every variable $x\in X$ occurring in $W$
occurs exactly twice (where occurrences of both $x$ and $x^{-1}$ are counted).
For a word $W \in G \ast F_{X}$ by $\Var(W) \subseteq X$ we denote the set of all variables
occurring in $W$ (again, occurrences of $x^{\pm1}$ are counted as occurrences of a variable $x$).

We denote $\Autf_G(G \ast F_{X})$ the group of {\em finitely supported $G$-automorphisms} of $G*F_X$,
i.e., automorphisms $\phi \in \Aut(G \ast F_{X})$ which are identical on $G$ and change finitely many elements
of $X$.
We say that two words $V,W \in G*F_X$ are {\em equivalent}
if there is an automorphism $\phi \in \Autf_G(G \ast F_{X})$
such that $\phi(V)$ is conjugate to ~$W$.
Clearly, if $V$ and $W$ are equivalent then
equation $V=1$ has a solution if and only if equation $W=1$ has a solution.

It is well known that every quadratic word is equivalent
to a word of one of the following forms:
\begin{equation} \label{eq:classical-standard}
  \begin{array}{c}
    \strut [x_1,y_1] [x_2,y_2] \ldots [x_g,y_g] \qquad (g \ge 0), \\[0.5ex]
    \strut [x_1,y_1] [x_2,y_2] \ldots [x_g,y_g] \, c_1 \, z_2^{-1} c_2 z_2 \dots z_{m}^{-1} c_{m} z_{m}
      \qquad (g \ge 0, \ m \ge 1), \\[0.5ex]
    x_1^2 x_2^2 \ldots x_g^2 \qquad (g > 0), \\[0.5ex]
    x_1^2 x_2^2 \ldots x_g^2 \, c_1 \, z_2^{-1} c_2 z_2 \dots z_{m}^{-1} c_{m} z_{m} \qquad (g > 0, \ m \ge 1),
  \end{array}
\end{equation}
where $x_i,y_i,z_i \in X$ are variables and $c_i \in G$
(see \cite{Comerford_Edmunds:1981} or \cite{Grigorchuk_Kurchanov:1992}).
With a slight change of these canonical forms
(introducing a new variable $z_1$, for technical convenience),
we call the following quadratic words $Q$ and the corresponding quadratic equations $Q=1$ {\em standard:}
\begin{gather*}
[x_1,y_1] [x_2,y_2] \ldots [x_g,y_g] \, z_1^{-1} c_1 z_1 \, z_2^{-1} c_2 z_2 \dots z_{m}^{-1} c_{m} z_{m}
  \qquad (g \ge 0, \ m \ge 0), \\
x_1^2 x_2^2 \ldots x_g^2 \, z_1^{-1} c_1 z_1 \, z_2^{-1} c_2 z_2 \dots z_{m}^{-1} c_{m} z_{m}
\qquad (g > 0, \ m \ge 0).
\end{gather*}

Words in the first and in the second series are called {\em standard orientable} and
{\em standard non-orientable}, respectively.
More generally, a quadratic word $Q$ (and a quadratic equation $Q=1$)
are called {\em orientable} if the two occurrences in $Q$ of each variable $x \in \Var(Q)$
have the opposite signs $x$ and $x^{-1}$ and {\em non-orientable} if there is a variable $x$
occurring in $Q$ twice with the same signs $x$ or $x^{-1}$.

The number $g$ is called the {\em genus} of a standard quadratic word $Q$.
The elements $c_1$, $\dots$, $c_m$ of $G$
occurring in $Q$ are called the {\em coefficients} of $Q$.

\begin{proposition} \label{pr:quadratic-standard}
Every quadratic word $Q$
is equivalent to a standard quadratic word $R$ which is
orientable if and and only if $Q$ is orientable. Moreover, $R$ and the equivalence automorphism
$\al \in \Autf_G(G*F_X)$ that sends $Q$ to a conjugate of $R$ can be computed effectively for a given $Q$.
\end{proposition}

\begin{proof}
Due to the reduction to the classical standard form \eqref{eq:classical-standard}
(the procedure in \cite{Comerford_Edmunds:1981} or in ~\cite{Grigorchuk_Kurchanov:1992} is
effective and preserves orientability),
it is enough to prove that removal of the variable ~$z_1$
in a standard quadratic word (in our sense) leads to an equivalent quadratic word. The following $G$-automorphism
does the job:
$$
  [x_1,y_1] \ldots [x_g,y_g] \cdot z_1^{-1} c_1 z_1 \cdot\ldots\cdot z_{m}^{-1} c_{m} z_{m} \\
    \xrightarrow{\phi}
  z_1^{-1} ([x_1,y_1] \ldots [x_g,y_g] \cdot c_1 z_2^{-1} c_2 z_2 \cdot\ldots\cdot z_{m}^{-1} c_{m} z_{m}) z_1
$$
where
$
  \phi = (x_i \mapsto z_1^{-1} x_i z_1, \ y_i \mapsto z_1^{-1} y_i z_1, \ i=1,\dots,g, \
	      z_i \mapsto z_i z_1, \ i=2,\dots,m ).
$
\end{proof}

\subsection{Equations with constraints modulo a subgroup}
Let $H$ be a normal subgroup of a group ~$G$. By $\pi_H$ we denote the canonical epimorphism
$G\rightarrow G/H$.

\begin{definition}
An {\em equation in $G$ with a constraint modulo $H$} is a pair $(W=1,\ga)$ where
$W \in G\ast F_X$ and $\ga$ is a map $\Var(W) \to G/H$. A solution of such an equation is a $G$-homomorphism
$\al: G\ast F_X \to G$ satisfying
$\al(W)=1$ and $\pi_H(\al(x))=\ga(x)$ for every variable $x \in \Var(W)$.
\end{definition}

This notion naturally extends to systems of equations in $G$. A constraint modulo $H$ for a system
of equations $\set{W_i=1}$ is a map
$$
  \ga: \bigcup_i \Var(W_i) \to G/H.
$$
A {\em solution} of a constrained system $(\set{W_i=1},\gamma)$ is a $G$-homomorphism
$\al: G\ast F_X \to G$ such that
$\al(W_i)=1$ for all $i$ and $\pi_H(\al(x))=\ga(x)$ for every $x \in \bigcup_i \Var(W_i)$.

If $Y \seq X$ is a set of variables then a map $\ga: Y \to G/H$ extends naturally to a group
homomorphism $G * F_Y \to G/H$ by defining $\ga(g) = \pi_H(g)$ for $g \in G$.
We use the same notation $\ga$ for this homomorphism (implicitly identifying the two maps).
In particular, a constraint ~$\ga$ for a system of equations $\set{R_i=1}$
is identified with the induced homomorphism $G * F_Y \to G/H$ where $Y = \bigcup_i \Var(W_i)$.

Observe that existence of a solution of a system of equations $(\set{R_i=1}, \ga)$ with a constraint ~$\ga$
automatically implies that $\ga(R_i) = 1$ for all $i$.

We introduce equivalence of constrained equations in the following way.

\begin{definition}
Equations $(W=1,\ga)$ and $(V=1,\ze)$ with constraints modulo $H$ are {\em equivalent}
if $\ga$ and $\ze$ can be extended to homomorphisms $\bar\ga, \bar\ze : G * F_X \to G/H$
so that for some $G$-automorphism $\phi \in \Autf_G(G * F_X)$, $\phi(W)$ is conjugate to $V$ and
$\bar\ze = \bar\ga \circ \phi$.
\end{definition}

The following simple observation shows that a constraint is naturally induced
by equivalence of equations.

\begin{lemma} \label{le:constraint-transformation}
Let $W$ and $V$ be equivalent words in $G*F_X$. Then for any constraint $\ga: \Var(W)\to G/H$
there exists another constraint $\ze: \Var(V)\to G/H$ such that equations $(W=1,\ga)$ and $(V=1,\ze)$
are equivalent. Given $W$, $\ga$ and a $G$-automorphism $\phi \in \Autf_G(G*F_X)$
sending $W$ to a conjugate of $V$, the constraint $\ze$ can be effectively computed.
\end{lemma}

\begin{proof}
To compute $\ze$, we first extend $\ga$ to a homomorphism $\bar\ga: G*F_X \to G/H$ in an arbitrary
way, then take $\bar\ze = \bar\ga \circ \phi$ and compute $\ze$ by restricting $\bar\ze$ to $F_X$.
Since $\phi$ is finitely supported, the procedure is effective.
\end{proof}

As an immediate consequence of the lemma and Proposition \ref{pr:quadratic-standard} we get
\begin{corollary} \label{co:constrained-standard}
For any quadratic equation $(Q=1,\ga)$ with a constraint modulo $H$ there is an equivalent
equation $(S=1,\ze)$ where $S$ is a standard quadratic word equivalent to $Q$.
\end{corollary}

Assume that $W_1$ and $W_2$ are two words in $G \ast F_X$ and there is a variable $x \in X$
which occurs in each $W_i$ exactly once. Let
$$
  W_i = U_i x^{\ep_i} V_i \quad\text{where $\ep_i =\pm1$}.
$$
We can express ~$x$ in ~$W_2$ as $x = (V_2 U_2)^{-\ep_2}$ and then substitute the expression in $W_1$
obtaining a new word denoted $W_1 \#_x W_2$ in which $x$ no longer occurs:
$$
  W_1 \#_x W_2 = U_1 (V_2 U_2)^{-\ep_1 \ep_2} V_1.
$$
Sometimes we simply write $W_1 \# W_2$ if the choice of $x$ is irrelevant
(see also Remark \ref{rm:join-words-invariance}).
It is obvious that a system $\set{W_1=1, W_2=1}$ is solvable in $G$
if and only if a single equation $W_1 \#_x W_2  = 1$ is solvable in $G$. We will need a similar statement
for the case of equations with constraints.

\begin{lemma} \label{le:join-correctness}
Let $W_1,W_2 \in G \ast F_X$ and assume that a variable $x \in X$ occurs in each ~$W_i$ exactly once.
Let $(\set{W_1=1, W_2=1}, \ga)$ be a system of equations in $G$ with a constraint $\ga$ modulo ~$H$
and $\ga(W_i) = 1$ for $i=1,2$.
Then this system has a solution if and only if the equation $(W_1 \#_x W_2 = 1,\ga')$
has a solution where $\ga'$ is the restriction of $\ga$ on $\Var(W_1 \#_x W_2)$.
\end{lemma}

\begin{proof}
The ``only if'' part is obvious. For the ``if'' part, we use the condition $\ga(W_i) = 1$
which implies that any solution $\al'$ of the constrained equation
$(W_1 \#_x W_2 = 1,\ \ga')$ extends to a solution of
the system $\set{W_1=1, W_2=1}$ with $\pi_H(\al(x)) = \ga(x)$.
\end{proof}

\begin{remark} \label{rm:join-words-invariance}
It is easy to see that if $y$ is another variable which occurs in either $W_1$ and $W_2$ exactly
once then $W_1 \#_y W_2$ and $W_1 \#_x W_2$ are equivalent. However, we do not need this fact
and the notation $W_1 \# W_2$ means a particular choice of a variable $x$ which is clear from the
context.
\end{remark}

\section{Splitting equations} \label{se:splitting}

\subsection{Splitting words in $\Ga*F_X$}

Let $W = 1$ be an equation in $\Ga$. If we substitute the values of a solution to $W$ and apply
the splitting homomorphism $\psi$ then we get two new equalities. These equalities lead
in a natural way to a system
$
  \set{W_0=1,\  W_1=1}
$
of two equations in ~$\Ga$ formally defined below in this section.
The main idea of splitting an equation is that we get a new equivalent system which,
in a certain sense, is simpler than the initial equation. Equivalence, however, cannot be achieved
in a straightforward way. An obstruction appears because the image of ~$\St_\Ga(1)$ under $\psi$ is
a proper subgroup of $\Ga \times \Ga$ and, in general, a solution of the system $\set{W_0=1, W_1=1}$
can not be lifted to a solution of $W=1$. This is the reason why we engage equations with constraints
modulo ~$K$: since we have $\psi(K) \supset K\times K$, for constrained equations the transition
from $W=1$ to $\set{W_0=1,\  W_1=1}$ is equivalent (see Corollary \ref{co:splitting-reduction}).

Starting from this point,
we consider only equations in $\Ga$ with constraints modulo $K$
(often omitting mentioning the constraints).
Since $K$ is a subgroup of $\Ga$ of finite index,
any equation in $\Ga$ is reduced to a finite disjunction of equations with constraints modulo $K$.

On the set of words $W \in \Ga * F_X$ we define two maps $\Psi_0$ and $\Psi_1$ which simulate
application of the homomorphisms $\psi_0$ and $\psi_1$ after substituting values of the variables in $W$.
Since $\psi$ is defined on the subgroup $\St_\Ga(1)$ of $\Ga$ of index ~2,
$\Psi_0(W)$ and $\Psi_1(W)$ depend
on the predefined cosets modulo $\St_\Ga(1)$ of all values of variables occurring in ~$W$.
We observe that a constraint modulo ~$K$ determines these cosets in
a unique way.
For this reason, we formally define maps ~$\Psi_i$ with respect to a given constraint
$\ga : \Var(W) \to \Ga/K$ (though denoting them $\Psi_i$ by abuse of notations).

Given a constraint $\ga: Y \to \Ga/K$ on a set of variables $Y \seq X$, we use the notation
$\si_\ga$ for the induced group homomorphism
$$
  \si_\ga: \Ga * F_Y \to \Ga / \St_\Ga(1)
$$
into the group $\Ga / \St_\Ga(1)$ of order 2
which gives the coset mod $\St_\Ga(1)$ of every word $U \in \Ga * F_Y$.

For an element $u\in \Ga$, let $\bar u$ denote the closest element in $\St_\Ga(1)$ defined by
$$
  \bar u =
  \begin{cases}
  u & \mbox{if } u\in \St_\Gamma(1),\\
  ua & \mbox{otherwise.}\\
  \end{cases}
$$
For each variable $x\in X$ we introduce two variables $x_0$ and $x_1$ which we call
the {\em descendants} of $x$. Since we operate on a single set of variables $X$
(and the splitting procedure will be applied to an equation recursively) we may formally assume
that $X$ is partitioned into two infinite disjoint sets $X_0$ and $X_1$ and two
bijections $X \to X_0$, $X \to X_1$ are fixed which provide the descendants of $x \in X$.

Now, given a word
    $$W = u_1 u_2 \dots u_k \in \Ga \ast F_X, \quad u_i \in \Ga \cup X^{\pm1},$$
and a constraint $\ga: \Var(W) \to \Ga/K$ we define a word:
    $$\Psi_0(W) = v_1 v_2 \dots v_k \in \Ga*F_X,$$
where for $u_i\in \Ga$,
$$
  v_i =
\begin{cases}
 \psi_0(\bar u_i) & \text{if } \si_\ga(u_1\dots u_{i-1}) =1 \\
 \psi_0(a \bar u_i a) & \text{if } \si_\ga(u_1\dots u_{i-1}) \ne 1\\
\end{cases}
$$
and for $u_i = x^\ep \in X^{\pm1}$,
$$
  v_i =
\begin{cases}
 x_0 & \text{if } \si_\ga(u_1\dots u_{i-1}) =1 \\
 x_1& \text{if } \si_\ga(u_1\dots u_{i-1}) \ne 1\\
\end{cases}
\quad \text{for } \ep=1, \qquad
  v_i =
\begin{cases}
 x_0^{-1} & \text{if } \si_\ga(u_1\dots u_{i}) =1 \\
 x_1^{-1} & \text{if } \si_\ga(u_1\dots u_{i}) \ne 1\\
\end{cases}
\quad \text{for } \ep=-1.
$$
Similarly one defines $\Psi_1(W)$ by taking $\psi_1$ instead of $\psi_0$ in the definition of $v_i$
for $u_i\in \Ga$
and interchanging $x_0$ and $x_1$ in the definition of $v_i$ for $u_i \in X^{\pm1}$.
We denote also
    $$\Psi(W) = (\Psi_0(W),\: \Psi_1(W)).$$

Note that in the definition of $\Psi_i(W)$ we do not assume that $\si_\ga(W) = 1$
(i.e.\ that $W$ defines an element in $\St_\Ga(1)$ after substituting values for
all variables) and thus $\Psi_i(W)$ is defined for {\em any} word
$W \in \Ga*F_X$. In particular, we have a function $\Psi:\Ga * F(X) \rightarrow (\Ga * F(X)) \times (\Ga * F(X))$.
Note also that $\Psi_i(W) = \Psi_i(Wa)$ for any $W$,
which can be seen directly from the definition.

Let $W\in\Ga\ast F(X)$ and $\ga:\Var(W)\to\Ga/K$ be a constraint on $\Var(W)$
(remember that $\Psi(W)$ is formally defined with respect to a given $\ga$).
For any $\Ga$-homomorphism $\al: \Ga * F_{\Var(W)} \to \Ga$
 we can define the induced map
$\al_*: \Ga * F_{\Var(\Psi_0(W)) \cup \Var(\Psi_1(W))} \to \Ga$ by
$$
  \al_*(x_i) =  \psi_i (\overline{\al(x)}) \qquad \text{for $x \in \Var(W)$ and $i=0,1$}.
$$
The next proposition follows from the construction by induction on the length of $W$.

\begin{proposition}[The main property of $\Psi$] \label{pr:splitting-correct}
For any $\Ga$-homomorphism $\al: \Ga * F_{\Var(W)} \to \Ga$ satisfying the constraint ~$\ga$
(that is, $\pi_K(\al(x)) = \ga(x)$ for any $x \in \Var(W)$)
we have $\psi_i(\overline{\al(W)}) = \al_*(\Psi_i(W))$
($i=0,1$).
\qed
\end{proposition}

We are in position to define splitting of an equation in $\Ga$ with a constraint modulo ~$K$.
Since the images $\psi_i(gK)$ of a coset $gK$ do not belong to a unique coset modulo $K$,
a constraint modulo $K$ generates a family of constraints under splitting.
To define this family, we use a notation $\bar g$ for an element $g \in \Ga/K$ which plays
the role of ``the closest element in the stabilizer $\St_1(\Ga)$'' (similar to the case of
notation $\bar g$ for $g \in \Ga$):
$$
  \bar g =
  \begin{cases}
    g &\text{if}\ g \in \St_1(\Ga)/K, \\
    g \, \pi_K(a) &\text{otherwise,}
  \end{cases}
$$
where $\pi_K(a)$ denotes the natural image of $a$ in $\Ga/K$.

\begin{definition} \label{de:splitting-constraint}
Given a word $W \in \Ga*F_X$ and a map $\ga: \Var(W)\to \Ga/K$,
we define a set $\CV_{W,\ga}$ of maps $\ze : \Var(\Psi_0(W)) \cup \Var(\Psi_1(W)) \to \Ga/K$:
\begin{equation}\label{eq:S_gamma}
  \CV_{W,\ga} = \setof{\ze}
  {\om(\ze(x_0),\ze(x_1)) = \overline{\ga(x)}
    \text{ for all $x \in \Var(W)$}}
\end{equation}
where $\om$ is given in Proposition \ref{pr:psi-via-K}.
\end{definition}

An immediate consequence of Propositions \ref{pr:splitting-correct}
and \ref{pr:psi-via-K} is the following corollary.

\begin{corollary}[The splitting reduction] \label{co:splitting-reduction}
Let $(W=1,\ga)$ be an equation in $\Ga$ and $\si_\ga(W)=1$.
Then $(W=1,\ga)$ is solvable if and only if
the system $(\set{\Psi_0(W)=1,\: \Psi_1(W)=1},\ \ze)$ is solvable for some $\ze\in \CV_{W,\ga}$.
\qed
\end{corollary}

\subsection{Splitting quadratic equations} \label{ss:quadratic-splitting}
In this subsection, we apply
$\Psi$ to standard quadratic equations in $\Ga$.

It follows from the definition of $\Psi_i$ that for any $U,V \in \Ga * F_X$:
$$
    \Psi_i (U \cdot V) =
  \begin{cases}
  \Psi_i (U) \cdot \Psi_i (V) & \text{if } \si_\ga(U) = 1, \\
  \Psi_i (U) \cdot \Psi_{1-i} (V) & \text{if } \si_\ga(U) \ne 1.
  \end{cases}
$$
Hence the image of a standard quadratic word under $\Psi_i$ is factored into blocks
of the form $\Psi_i([x,y])$, $\Psi_i(x^2)$ and $\Psi_j(z^{-1} c z)$, $j=0,1$.
(Note that $\si_\ga([x,y]) = \si_\ga(x^2) = 1$.)
We write explicit expressions for these factors (we assume that commutators $[x,y]$
are written as $x^{-1}y^{-1}xy$):
\begin{align*}
  \Psi_0([x,y]) &= x_0^{-1} y_0^{-1} x_0 y_0, &
  \Psi_1([x,y]) &= x_1^{-1} y_1^{-1} x_1 y_1  && \text{if } \si_\ga(x) = \si_\ga(y)=1, \\
  \Psi_0([x,y]) &= x_1^{-1} y_1^{-1} x_1 y_0, &
  \Psi_1([x,y]) &= x_0^{-1} y_0^{-1} x_0 y_1  && \text{if } \si_\ga(x) \ne 1,\ \si_\ga(y)=1, \\
  \Psi_0([x,y]) &= x_0^{-1} y_1^{-1} x_1 y_1, &
  \Psi_1([x,y]) &= x_1^{-1} y_0^{-1} x_0 y_0  && \text{if } \si_\ga(x) = 1,\ \si_\ga(y) \ne 1, \\
  \Psi_0([x,y]) &= x_1^{-1} y_0^{-1} x_0 y_1, &
  \Psi_1([x,y]) &= x_0^{-1} y_1^{-1} x_1 y_0  && \text{if } \si_\ga(x),\si_\ga(y) \ne 1,
\end{align*}
\begin{align*}
  \Psi_0(x^2) &= x_0^2, &
  \Psi_1(x^2) &= x_1^2, && \text{if } \si_\ga(x) = 1, \\
  \Psi_0(x^2) &= x_0 x_1, &
  \Psi_1(x^2) &= x_1 x_0, && \text{if } \si_\ga(x) \ne 1,
\end{align*}
and finally,
\begin{align*}
  \Psi_0(z^{-1} c z) &= z_0^{-1} c_0 z_0, &
  \Psi_1(z^{-1} c z) &= z_1^{-1} c_1 z_1 &&  \text{if } c \in St_\Ga(1), \ \si_\ga(z)=1, \\
  \Psi_0(z^{-1} c z) &= z_0^{-1} c_0 z_1, &
  \Psi_1(z^{-1} c z) &= z_1^{-1} c_1 z_0 &&  \text{if } c \notin St_\Ga(1), \ \si_\ga(z)=1,\\
  \Psi_0(z^{-1} c z) &= z_1^{-1} c_1 z_1, &
  \Psi_1(z^{-1} c z) &= z_0^{-1} c_0 z_0 &&  \text{if } c \in St_\Ga(1), \ \si_\ga(z) \ne 1, \\
  \Psi_0(z^{-1} c z) &= z_1^{-1} c_1 z_0, &
  \Psi_1(z^{-1} c z) &= z_0^{-1} c_0 z_1 &&  \text{if } c \notin St_\Ga(1), \ \si_\ga(z) \ne 1.
\end{align*}
where
$$
  c_i = \psi_i (c), \quad i=0,1.
$$

For a standard quadratic word $Q$, denote by $C(Q)$ the set of coefficients of $Q$.

\begin{lemma}\label{le:quadratic_split}
Let $(Q=1,\gamma)$ be a standard quadratic equation in $\Ga$ 
and $\Psi(Q) = (Q_0,Q_1)$. Then the following assertions are true.
\begin{enumerate}
\item
$\Var(Q_0) \cap \Var(Q_1) = \emptyset$ if and only if
$C(Q) \subseteq \St_\Ga(1)$ and either
$\si_\ga(x_i) = \si_\ga(y_i) =1$ for every commutator $[x_i,y_i]$ in the commutator part of ~$Q$
(if ~$Q$ is standard orientable) or $\si_\ga(x_i) =1$ for every square $x_i^2$ in the square part of ~$Q$
(if ~$Q$ is standard non-orientable).
    \item
If $\Var(Q_0) \cap \Var(Q_1) = \emptyset$, then both $Q_0$ and $Q_1$ are standard quadratic words
of the same genus $g$ and the same orientability as of $Q$.
Furthermore,
$$
  C(Q_i) = \{\psi_i(c) \mid c\in C(Q),\ \psi_i(c)\ne 1\}.
$$
    \item
If $x\in \Var(Q_0) \cap \Var(Q_1)$, then $Q_0\#_x Q_1$ is a quadratic word. If $Q$
is orientable then $Q_0\#_x Q_1$ is also orientable.
\end{enumerate}
\end{lemma}

\begin{proof}
Straightforward verification.
\end{proof}

In Lemma \ref{le:split-computation} we collect all necessary computations which we will use later
to describe the standard form of the quadratic word $Q_0\# Q_1$ in the case
$\Var(Q_0) \cap \Var(Q_1) \ne \emptyset$.
We write $U \sim V$ for equivalence of words $U,V \in \Ga * F_X$.

\begin{lemma} \label{le:split-computation}
Let $Q$ be a quadratic word, $x_0,x_1,y_0,y_1,z_1,z_2,z_3,z_4$ be variables
not occurring in $Q$, and $c_1,c_2,c_3,c_4\in\Ga$.
The following holds:
\begin{enumerate}
\item
If $Q = U V$ then
$$
  U [x_0, y_0] V  \sim [x_0, y_0] Q, \quad
  U x_0^2 V  \sim x_0^2 Q \quad\text{and}\quad
  U z_1^{-1} c_1 z_1 V \sim Q z_1^{-1} c_1 z_1.
$$
\item
If $Q = U V W$ and $(R,S)$ is one of the pairs
$$
  (x_1^{-1} y_1^{-1} x_1 y_0, \ x_0^{-1} y_0^{-1} x_0 y_1), \quad
  (x_0^{-1} y_1^{-1} x_1 y_1, \ x_1^{-1} y_0^{-1} x_0 y_0) \quad\text{or}\quad
  (x_1^{-1} y_0^{-1} x_0 y_1, \ x_0^{-1} y_1^{-1} x_1 y_0)
$$
then
$$
  URVSW \sim [x_0, y_0] [x_1, y_1] Q.
$$
\item
If $Q = U V W$ then
$
  U x_0 x_1 V x_1 x_0 W \sim x_0^2 x_1^2 Q.
$
\item
If $Q = U V W$  then
$$
  U \cdot z_1^{-1} c_1 z_2 \cdot z_3^{-1} c_3 z_4 \cdot V \cdot z_2^{-1} c_2 z_1 \cdot z_4^{-1} c_4 z_3 \cdot
  W
    \sim [x_0 ,y_0] Q  \cdot z_1^{-1} c_1 c_2 z_1 \cdot z_2^{-1} c_3 c_4 z_2.
$$
\item
If $(R,S)$ is one of the pairs in (ii), then
$R \# S \sim [x_0, x_1]$.
\item
$
  z_1^{-1} c_1 z_2 \cdot z_3^{-1} c_3 z_4 \: \# \: z_2^{-1} c_2 z_1 \cdot z_4^{-1} c_4 z_3 \sim
  z_1^{-1} c_1 c_2 z_1 \cdot z_2^{-1} c_3 c_4 z_2.
$
\end{enumerate}
\end{lemma}

\begin{proof}
Straightforward computations.
(i):
\begin{gather*}
  U [x_0, y_0] V
  \xrightarrow{(x_0 \mapsto U^{-1} x_0 U, \ y_0 \mapsto U^{-1} y_0 U)}
  [x_0, y_0] U V, \\
  U x_0^2 V
  \xrightarrow{(x_0 \mapsto U^{-1} x_0 U)}
  x_0^2 U V, \\
  U z_1^{-1} c_1 z_1 V
  \xrightarrow{(z_1 \mapsto z_1 V^{-1})}
  U V z_1^{-1} c_1 z_1.
\end{gather*}
To prove (ii), assume $R = x_1^{-1} y_1^{-1} x_1 y_0$ and $S = x_0^{-1} y_0^{-1} x_0 y_1$. Then:
\begin{align*}
  &U x_1^{-1} y_1^{-1} x_1 y_0 V x_0^{-1} y_0^{-1} x_0 y_1 W
    &\xrightarrow{(x_0 \mapsto x_0 V, \ x_1 \mapsto V^{-1}x_1, \ y_1 \mapsto V^{-1} y_1 V)}\quad
  &U x_1^{-1} y_1^{-1} x_1 y_0 x_0^{-1} y_0^{-1} x_0 y_1 V W \\
    &&\xrightarrow{(x_i \mapsto U^{-1}x_i U,\ y_i \mapsto U^{-1} y_i U),\ i=0,1}\quad
  &x_1^{-1} y_1^{-1} x_1 y_0 x_0^{-1} y_0^{-1} x_0 y_1 U V W \\
    &&\xrightarrow{(x_0 \mapsto y_1 x_0 y_1^{-1}, \ y_0 \mapsto y_1 y_0 y_1^{-1})}\quad
  &[x_0, y_1] [y_0^{-1}, x_1] UVW \\[1mm]
    && \sim \quad
  &[x_0, y_0] [x_1, y_1] UVW.
\end{align*}
The other two cases for $(R,S)$ are similar.

(iii):
The quadratic word $U x_0 x_1 V x_1 x_0 W$ can be modified as follows:
\begin{alignat*}{2}
  && U x_0 x_1 V x_1 x_0 W
    \xrightarrow{(x_0 \mapsto x_0 V, \ x_1 \mapsto V^{-1} x_1)}\quad
  &U x_0 x_1^2 x_0 V W \\
    &&\xrightarrow{(x_0 \mapsto U^{-1} x_0 U, \ x_1 \mapsto U x_1 U^{-1})}\quad
  &x_0 x_1^2 x_0 U V W \\
    &&\xrightarrow{(x_0 \mapsto x_0 x_1^{-2}, \ x_1 \mapsto x_1^{-1})}\quad
  &x_0^2 x_1^2 U V W.
\end{alignat*}

(iv): The quadratic word $U z_1^{-1} c_1 z_2 z_3^{-1} c_3 z_4 V  z_2^{-1} c_2 z_1 z_4^{-1} c_4 z_3 W$
can be modified as follows:
\begin{align*}
   U z_1^{-1} c_1 z_2 z_3^{-1} c_3 z_4 V  z_2^{-1} c_2 z_1 z_4^{-1} c_4 z_3 W \hspace{-2em} \\
    &\xrightarrow{(z_2 \mapsto z_2 V, \ z_3 \mapsto z_3 V)}
  &&U z_1^{-1} c_1 z_2 z_3^{-1} c_3 z_4 z_2^{-1} c_2 z_1 z_4^{-1} c_4 z_3 VW \\
    &\xrightarrow{(z_1 \mapsto z_1 U, \ z_4 \mapsto z_4 U)}
  &&z_1^{-1} c_1 z_2 z_3^{-1} c_3 z_4 U z_2^{-1} c_2 z_1 z_4^{-1} c_4 z_3 VW \\
    &\xrightarrow{(z_2 \mapsto z_2 U, \ z_3 \mapsto z_3 U)}
  &&z_1^{-1} c_1 z_2 z_3^{-1} c_3 z_4 z_2^{-1} c_2 z_1 z_4^{-1} c_4 z_3 UVW
\end{align*}
A reduction of $z_1^{-1} c_1 z_2 z_3^{-1} c_3 z_4 z_2^{-1} c_2 z_1 z_4^{-1} c_4 z_3$ to
the standard form gives
$$
  z_1^{-1} c_1 z_2 z_3^{-1} c_3 z_4 z_2^{-1} c_2 z_1 z_4^{-1} c_4 z_3 Q
  \sim
  [x_0 ,y_0]  z_1^{-1} c_1 c_2 z_1 z_2^{-1} c_3 c_4 z_2 Q.
$$
We then move the factor $z_1^{-1} c_1 c_2 z_1 z_2^{-1} c_3 c_4 z_2$ to the end of $Q$ by (i).
Equivalences (v) and (vi) are similar.
\end{proof}

\begin{proposition}[Non-disjoint orientable case]\label{pr:image-joint}
Let $(Q=1,\ga)$ be a quadratic equation where
    $$Q = [x_1,y_1] [x_2,y_2] \ldots [x_g,y_g] \cdot z_1^{-1} c_1 z_1 \cdot\ldots\cdot z_{m}^{-1} c_{m} z_{m}$$
is a standard orientable quadratic word and $\si_\ga(Q) = 1$. Let $\Psi(Q) = (Q_0,Q_1)$.
Assume that $\Var(Q_0) \cap \Var(Q_1) \ne \emptyset$. Then $Q_0 \# Q_1$ is equivalent to
a standard quadratic word:
$$
  R = [x_1,y_1] [x_2,y_2] \ldots [x_h,y_h] \cdot z_1^{-1} d_1 z_1 \cdot\ldots\cdot z_{l}^{-1} d_{l} z_{l}
$$
satisfying the following:
\begin{enumerate}
    \item
$h = 2g+\frac{1}{2} \de(Q)-1$, where $\de(Q)$ is the cardinality of the set $\{i\mid c_i\notin \St_\Ga(1)\}$;
    \item
$C(R) = \cup_{i=1}^m K_i \setminus \{1\}$, where
$$
  K_i =
\begin{cases}
\set{\psi_0(c_i),\: \psi_1(c_i)} & \text{if } c_i \in \St_\Ga(1), \\
\set{\psi_0(c_ia)\psi_1(c_ia),\: \psi_1(c_ia)\psi_0(c_ia)} & \text{if } c_i \notin \St_\Ga(1).
\end{cases}
$$
\end{enumerate}
\end{proposition}

\begin{proof}
The assumption $\gamma(Q)\in \St_\Ga(1)$ implies that the number $\delta(Q)$ is even.
By Lemma \ref{le:quadratic_split}(i), we have $\si_\ga(x_i) \ne 1$ or $\si_\ga(y_i) \ne 1$ for some
commutator $[x_i,y_i]$ in ~$Q$ or $c_j \notin St_\Ga(1)$ for some $j$. We compute the standard form of $Q_0\#
Q_1$.

{\em Case\/} 1:  $\si_\ga(x_i) \ne 1$ or $\si_\ga(y_i) \ne 1$ for some $i$.
Let $Q = U [x_i,y_i] V$. Then
\begin{align*}
    Q_0 &= U_0 x_{i1}^{-1} y_{i1}^{-1} x_{i1} y_{i0} V_0, &
    Q_1 &= U_1 x_{i0}^{-1} y_{i0}^{-1} x_{i0} y_{i1} V_1 &&
    \text{if } \si_\ga(x_i) \ne 1, \ \si_\ga(y_i) = 1, \\
    Q_0 &= U_0 x_{i0}^{-1} y_{i1}^{-1} x_{i1} y_{i1} V_0, &
    Q_1 &= U_1 x_{i1}^{-1} y_{i0}^{-1} x_{i0} y_{i0} V_1 &&
    \text{if } \si_\ga(x_i) \ne 1, \ \si_\ga(y_i) = 1, \\
    Q_0 &= U_0 x_{i1}^{-1} y_{i0}^{-1} x_{i0} y_{i1} V_0, &
    Q_1 &= U_1 x_{i0}^{-1} y_{i1}^{-1} x_{i1} y_{i0} V_1 &&
    \text{if } \si_\ga(x_i) \ne 1, \ \si_\ga(y_i) = 1
\end{align*}
where $U_k = \Psi_k(U)$, $V_k = \Psi_k(V)$ for $k=0,1$.
We have the corresponding cases for $Q_0 \# Q_1$:
\begin{align*}
  Q_0 \#_{y_{i0}} Q_1 &= U_0 x_{i1}^{-1} y_{i1}^{-1} x_{i1} x_{i0} y_{i1} V_1 U_1 x_{i0}^{-1} V_0, \quad \text{or}
  \\
  Q_0 \#_{x_{i1}} Q_1 &= U_0 x_{i0}^{-1} y_{i1}^{-1} y_{i0}^{-1} x_{i0} y_{i0} V_1 U_1 y_{i1} V_0, \quad \text{or}
  \\
  Q_0 \#_{x_{i0}} Q_1 &= U_0 x_{i1}^{-1} y_{i0}^{-1} y_{i1}^{-1} x_{i1} y_{i0} V_1 U_1 y_{i1} V_0.
\end{align*}
Assume that $\si_\ga(x_i) \ne 1$ and $\si_\ga(y_i) = 1$ (the other two cases are similar).
Using Lemma \ref{le:split-computation} we reduce $Q_0 \#_{y_{i0}} Q_1$ to a standard form $R$:
\begin{itemize}
\item
By statements (i) and (ii) of the lemma, collect words $\psi_k([x_j,y_j])$ for each commutator $[x_j,y_j]$
in $UV$ to the left; each commutator $[x_j,y_j]$ in $UV$ contributes then two commutators to $R$.
\item
By statement (i) of the lemma, collect
words $\psi_k(z_j^{-1}c_j z_j)$ for each coefficient factor $z_j^{-1}c_j z_j$
with $c_j\in\St_\Ga(1)$ to the right; each factor $z_j^{-1}c_j z_j$ contributes to $R$ at most two coefficient
factors
of a similar form (if $\psi_k(c_j) = 1$ then the factor with $\psi_k(c_j)$ disappears).
\item
By statement (vi) of the lemma,
collect words $\psi_k(z_j^{-1}c_j z_j)$ for the remaining coefficient factors $z_j^{-1}c_j z_j$
with $c_j\notin\St_\Ga(1)$ to the right (they are now paired as in the left-hand side of the equivalence in
(vi)).
Each pair of factors $z_j^{-1}c_j z_j$ with $c_j\notin\St_\Ga(1)$ contributes one commutator and at most one
coefficient factor to $R$;
\item Finally, replace the remaining non-reduced subword with a commutator by Lemma
    \ref{le:split-computation}(v).
\end{itemize}

{\em Case\/} 2: $c_j \notin St_\Ga(1)$ for some $j$. Let $Q = U z_j^{-1} c_j z_j V$. Then
\begin{gather*}
  Q_0 = U_0 z_{j0}^{-1} c_{j0} z_{j1} V_1, \quad
  Q_1 = U_1 z_{j1}^{-1} c_{j1} z_{j0} V_0 \quad
 \text{if } \si_\ga(z_j) = 1 \\
  Q_0 = U_0 z_{j1}^{-1} c_{j1} z_{j0} V_1, \quad
  Q_1 = U_1 z_{j0}^{-1} c_{j0} z_{j1} V_0 \quad
 \text{if } \si_\ga(z_j) \ne 1
\end{gather*}
where $U_k = \Psi_k(U)$, $V_k = \Psi_k(V)$, $c_{jk} = \psi_k(\bar{c}_j)$, $k=0,1$.
Up to re-enumeration of variables and coefficients, we may assume that $\si_\ga(z_j) = 1$.
In this case
$$
  Q_0 \#_{z_{j0}} Q_1 = U_0 V_0 U_1 z_{j1}^{-1} c_{j1}  c_{j0} z_{j1} V_1.
$$
Then we proceed similarly to Case 1.

Statements (i) and (ii) of Proposition \ref{pr:image-joint} now easily follow from the reduction process
and right hand sides of the equivalences in Lemma \ref{le:split-computation}(i,iv,vi).
\end{proof}

\begin{proposition}[Non-disjoint non-orientable case]\label{pr:image-joint-nonorient}
Let $(Q=1,\ga)$ be a quadratic equation where
    $$Q = x_1^2 x_2^2 \ldots x_g^2 \cdot z_1^{-1} c_1 z_1 \cdot\ldots\cdot z_{m}^{-1} c_{m} z_{m}$$
is a standard non-orientable quadratic word and $\si_\ga(Q) = 1$.
Let $\Psi(Q) = (Q_0,Q_1)$ and $\Var(Q_0) \cap \Var(Q_1) \ne \emptyset$.
Then $Q_0 \# Q_1$ is equivalent to a standard quadratic word
(which is non-orientable if $g>0$ and orientable otherwise)
$$
  R = x_1^2 x_2^2 \ldots x_h^2 \cdot z_1^{-1} d_1 z_1 \cdot\ldots\cdot z_{l}^{-1} d_{l} z_{l}
$$
satisfying the following:
\begin{enumerate}
    \item
$h = 2g+\de(Q)-2$;
    \item
$C(R) = \set{d_1,d_2,\dots,d_l}$ is the same as in Proposition ~\ref{pr:image-joint}.
\end{enumerate}
\end{proposition}

\begin{proof}
Similar to the proof of Proposition \ref{pr:image-joint}.
There is a slight difference in computing the genus $h$: in case of a single square $Q=x_0^2$ we get:
$$
  R = x_0 x_1 \# x_1 x_0 =1
$$
and each commutator coming from the coefficients by Lemma \ref{le:split-computation}(iv) contributes
$2$ to ~$h$ by the equivalence
$x^2 [y, z] \sim z^2 y^2 z^2$.
\end{proof}

We summarize properties of the splitting operation for constrained quadratic equations in ~$\Ga$
in the following proposition.

\begin{proposition} \label{pr:constrained-quadratic-splitting}
Let $(Q=1,\ga)$ be a standard quadratic equation in $\Ga$ with a constraint modulo $K$.
Assume that $\si_\ga(Q) = 1$ and let $\Psi(Q) = (Q_0, Q_1)$.
\begin{enumerate}
\item
Suppose that $\Var(Q_0) \cap \Var(Q_1) = \eset$. Then $Q_0$ and $Q_1$ are
standard quadratic words of the same genus and orientability as $Q$.
The coefficients of ~$Q_i$ are nontrivial elements $\psi_i(c_j)$, where $c_1,\dots,c_m$
are the coefficients of $Q$.
There are finitely many pairs of constraints $(\ga_{0j},\ga_{1j})$ such that
the equation $(Q=1, \ga)$ is solvable if and only if, for some $j$, both equations
$(Q_0=1, \ga_{0j})$ and $(Q_1=1, \ga_{1j})$ are solvable.

The set $\set{(\ga_{0j},\ga_{1j})}$ of pairs of constraints $\ga_{ij}$
is defined by restricting each constraint in $\CV_{Q,\ga}$ (see Definition \ref{de:splitting-constraint})
to $\Var(Q_0)$ and $\Var(Q_1)$. In other words, a pair $(\ga_0,\ga_1)$ belongs to this set if and only if
$$
  \om(\ga_0(x_{0}), \ga_1(x_{1})) = \overline {\ga(x)} \quad\text{for each } x \in \Var(Q),
$$
where $x_0,x_1$ are the descendants of a variable $x$ and $\om$
is given by Proposition ~\ref{pr:psi-via-K}.

\item
Suppose that $\Var(Q_1) \cap \Var(Q_2) \ne \eset$.
Then there is a standard quadratic word ~$R$ equivalent to $Q_0 \# Q_1$ and
finitely many constraints $\de_j: \Var(R) \to \Ga/K$ such that
the equation $(Q=1, \ga)$ is solvable if and only if, for some $j$, the equation $(R=1, \de_j)$ is solvable.
If $Q$ is orientable then $R$ is orientable. The genus and the coefficients of ~$R$ are as
in Propositions ~\ref{pr:image-joint} and ~\ref{pr:image-joint-nonorient}.

The set $\set{\de_j}$ is defined in the following way. Let $\phi\in \Autf_\Ga(\Ga*F_X)$
be a $\Ga$-automorphism sending $Q_0 \# Q_1$ to a conjugate of $R$.
We take the set $\CV_{Q,\ga}$ of constraints for $Q_0 \# Q_1$ defined in \eqref{eq:S_gamma},
and the subset $\CU$ of $\CV_{Q,\ga}$ of those $\ze \in \CV_{Q,\ga}$ which satisfy $\ze(Q_0)=\ze(Q_1)=1$.
Then for each $\ze \in \CU$, we take its restriction on $\Var(Q_0)\cup\Var(Q_1)$
and produce a constraint $\de: \Var(R) \to G/K$ using ~$\phi$ by Lemma ~\ref{le:constraint-transformation}.
\end{enumerate}

All the data provided by assertions (i) and (ii) can be effectively computed from the equation
$(Q=1,\ga)$.
\end{proposition}

\begin{proof}
Follows from Lemmas \ref{le:join-correctness}, \ref{le:quadratic_split},
Corollaries \ref{co:constrained-standard}, \ref{co:splitting-reduction} and
Propositions \ref{pr:image-joint}, \ref{pr:image-joint-nonorient}.
\end{proof}

\begin{remark} \label{rm:no-conjugation}
The transformation automorphism $\phi$ in
Proposition \ref{pr:constrained-quadratic-splitting}(ii) that
sends $Q_0 \# Q_1$ to its standard form $R$ can be chosen in such a way that
$\phi(Q_0 \# Q_1) = R$ without conjugation. This can be seen in a straightforward
way from the proofs of
Propositions \ref{pr:image-joint} and \ref{pr:image-joint-nonorient} and
the fact that conjugation in not needed in equivalences (v) and (vi) of Lemma \ref{le:split-computation}.
\end{remark}

\section{Solution of the Diophantine problem for quadratic equations} \label{se:diophantine}

In this section we prove Theorem \ref{thm:solvability}
by presenting an algorithm which for a given (unconstrained) quadratic equation $Q=1$ in $\Ga$
determines if the equation has a solution. The algorithm consists of Steps ~1--5 below.
To simplify notations, we assume that $Q$ is an orientable quadratic word (the non-orientable case is
literally the same, with commutators replaced by squares).

{\em Step\/} 1.
We reduce $Q$ to the standard form according to Proposition \ref{pr:quadratic-standard}.
Thus, from now on we write $Q$ as
$$
  Q = [x_1,y_1]\ldots[x_g,y_g] z_1^{-1} c_1 z_1 \ldots z_m^{-1} c_m z_m.
$$

{\em Step\/} 2.
We reduce the problem to constrained equations. For a given $Q$,
we write a finite list of all possible constraints $\ga_i: \Var(Q) \to \Ga/K$.
Then the equation $Q=1$ is solvable if and only if the constrained equation $(Q=1,\ga_i)$ is solvable
for some $i$.

We assume now that we are given a constrained standard quadratic equation $(Q=1,\ga)$.

{\em Step\/} 3. Given a standard equation $(Q=1,\ga)$, we start
recursive application of the splitting procedure described in Proposition
\ref{pr:constrained-quadratic-splitting}.
We use the following fact.

\begin{proposition}[Coefficient reduction] \label{pr:coefficient-reduction}
Let $(g_0,g_1,\dots)$ be a sequence
of elements in $\Ga$ satisfying the following condition
$$
g_{i+1} \in
\begin{cases}
\{\psi_0(g_i),\: \psi_1(g_i)\} & \mbox{ if } g_i\in \St_\Ga(1)\\
\{\psi_0(g_ia)\psi_1(g_ia),\: \psi_1(g_ia)\psi_0(g_ia)\} & \mbox{ if } g_i\notin \St_\Ga(1)\\
\end{cases}
$$
Then there exists $M = M(g_0)$ such that $|g_n|\le 3$ for every $n \ge M$.
In fact, one can take:
    $$M = 200+\log_{1.22} \max\set{1, |g_0|-200}.$$
\end{proposition}

\begin{proof}
Follows from Proposition 3.6 in \cite{Lysenok_Miasnikov_Ushakov:2010}.
\end{proof}

After applying the splitting operation at most $M$ times, we find a finite set $\CF$ of
systems of equations such that the solvability of $(Q=1,\ga)$ is equivalent to the solvability of
at least one system in $\CF$. Each system in $\CF$ is a finite set $\set{(Q_i=1,\ga_i)}$
of mutually independent quadratic equations $(Q_i=1,\ga_i)$ written in the standard form where
the length of each coefficient is at most 3.
Define a set:
$$
  \CS = \setof{g \in \Ga}{|g| \le 3}.
$$
Denote by $\CE_\CS$ the set of all standard orientable quadratic equations $(Q=1,\ga)$
with coefficients in $\CS$.
Now we may assume that we are given an equation $(Q=1,\ga)$ in $\CE_\CS$.

{\em Step\/} 4.
We fix a linear ordering on finite sets $\Ga/K$ and $\CS$.
Given an equation $(Q=1,\ga)$ in $\CE_\CS$, we transform it to the
{\em ordered form} according to the following lemma:

\begin{lemma}[Ordering factors]\label{le:ordering_terms}
For every equation $(Q=1,\ga)$ in $\CE_{\mathcal S}$, there exists (and can be effectively computed)
an equivalent equation $(Q=1,\ze)$ satisfying:
\begin{equation} \label{eq:ord1}
  (\ze(x_1),\ze(y_1)) \preceq (\ze(x_2),\ze(y_2)) \preceq \ldots \preceq (\ze(x_g),\ze(y_g))
\end{equation}
and
\begin{equation} \label{eq:ord2}
  (c_1,\ze(z_1)) \preceq (c_2,\ze(z_2)) \preceq \ldots \preceq (c_m,\ze(z_m))
\end{equation}
where ``$\preceq$'' is the lexicographic order induced by the orderings
on $\Ga/K$ and ~$\CS$.
\end{lemma}

\begin{proof}
If $(\ga(x_{i+1}),\ga(y_{i+1})) \prec (\ga(x_{i}),\ga(y_{i}))$ then applying to $Q$
an automorphism:
$$
(x_i \to [x_{i+1},y_{i+1}] x_i [x_{i+1},y_{i+1}]^{-1}, \quad
y_i \to [x_{i+1},y_{i+1}] y_i [x_{i+1},y_{i+1}]^{-1})
$$
swaps $[x_{i},y_{i}]$ and $[x_{i+1},y_{i+1}]$ and, possibly,
changes $\ga(x_{i})$ and $\ga(y_{i})$.
For the new equation, the sequence of pairs
$$
  \bigl((\ga(x_{1}),\ga(y_{1})), \ (\ga(x_{2}),\ga(y_{2})), \ \dots, \ (\ga(x_{g}),\ga(y_{g}))\bigr)
$$
is lexicographically smaller than that for $Q$.
Therefore, after applying a finite sequence of such automorphisms
we get an equation satisfying \eqref{eq:ord1}.

If $(c_{i+1},\ga(z_{i+1})) \prec (c_{i},\ga(z_{i}))$, then applying to $Q$ an automorphism
$$
  (z_i \to z_i \cdot z_{i+1}^{-1} c_{i+1}^{-1} z_{i+1})
$$
swaps $z_{i}^{-1} c_{i}^{-1} z_{i}$ and $z_{i+1}^{-1} c_{i+1}^{-1} z_{i+1}$ and, possibly,
changes $\ga(z_{i})$. For the new equation,
the sequence of pairs
$$
  \bigl( (c_{1},\ga(z_{1})), \ (c_{2},\ga(z_{2})), \ \dots, \ (c_{m},\ga(z_{m})) \bigr)
$$
is lexicographically smaller than that for $Q$. Therefore, a sequence of such
transformations stops in finitely many steps with an equation
satisfying also \eqref{eq:ord2}.
\end{proof}

{\em Step\/} 5.
Denote
$$
  \CB = (\Ga/K\times \Ga/K) \cup (\Ga/K \times \CS).
$$
Note that $\CB$ is finite since both $\Ga/K$ and $\CS$ are finite.
Every ordered equation $(Q=1,\ga)$ in ~$\CE_\CS$
can be encoded as a function $\la_{Q,\ga} \in \MN^\CB$ which associates
\begin{itemize}
    \item
to every pair $(g,h) \in \Ga/K\times \Ga/K$ the number of factors $[x_i,y_i]$ in $Q$
such that $\ga(x_i) = g$ and $\ga(y_i) = h$;
    \item
to every pair $(g,c) \in \Ga/K\times \CS$ the number of factors $z_i^{-1} c_i z_i$ in $Q$
such that $\ga(z_i) = g$ and $c_i = c$.
\end{itemize}
Let $\CP$ be a set of all functions $\la_{Q,\ga}$ encoding equations
$(Q=1,\ga)$ that have solutions. All we need to show is that $\CP$ is recursive.

We fix any set of representatives in $\Ga$ of all elements of $\Ga/K$, so for any
$h \in \Ga/K$ we have $\hat h \in \Ga$ with $\pi_K(\hat h) = h$.
Denote by $\mathit{Order}(g)$ the order of an element $g \in \Ga$
(it is finite since $\Ga$ is a 2-group, see Theorem 17 in \cite[Chapter VIII]{Harpe}).

Let $\CL \seq \MN^\CB$ be the set of all non-negative linear combinations
of the following functions $\mu_{g,h}$ and
$\nu_{g,c}$ where $(g,h)$ and $(g,c)$ run over $\Ga/K\times \Ga/K$ and $\Ga/K \times \CS$
respectively:
$$
  \mu_{g,h} \bigl((g,h)\bigr) = \mathit{Order}([\hat g, \hat h]), \quad \mu(u) = 0
      \quad\text{for all other } u \in \CB
$$
and
$$
  \nu_{g,c} \bigl((g,c)\bigr) = \mathit{Order}(c), \quad \mu(u) = 0
      \quad\text{for all other } u \in \CB.
$$

\begin{lemma} \label{le:P-is-cone}
$\CP + \CL \seq \CP$.
\end{lemma}

\begin{proof}
It is enough to prove that $\CP + \xi \seq \CP$ where $\xi$ is either $\mu_{g,h}$ or
$\nu_{g,c}$.
Let $(Q=1,\ga)$ and $(Q_1=1,\ga_1)$ be two
equations such that $\la_{Q_1,\ga_1} = \la_{Q,\ga} + \mu_{g,h}$. Then $Q_1$ is obtained from ~$Q$ by
inserting (at an appropriate place) the product $[x_1, y_1]\dots [x_r, y_r]$ of
$r=\textit{Order}(\hat g,\hat h)$ commutators $[x_i,y_i]$
and defining the constraint ~$\ga_1$ on the new variables by
$$
  \ga_1(x_1) = \ga_1(x_2) = \dots = \ga_1(x_r) = g \quad\text{and}\quad
  \ga_1(y_1) = \ga_1(y_2) = \dots = \ga_1(y_r) = h.
$$
If $\al$ is a solution of $(Q=1,\ga)$ then we can define a solution $\al_1$ of $(Q_1=1,\ga_1)$
by extending ~$\al$ on the new variables $\set{x_i, y_i}$ by setting $\al_1(x_i) = \hat g$
and $\al_1(y_i) = \hat h$ for all $i$.
The case when $\xi=\nu_{g,c}$ is similar.
\end{proof}

\begin{lemma}\label{le:convex_union}
Let $R$ be a subset of $\MN^n$ such that $R+\MN^n \subseteq R$.
Then there exist finitely many vectors
$v_1,\ldots,v_m \in R$ such that
    $$R = (v_1+\MN^n) \cup \ldots \cup (v_m+\MN^n).$$
\end{lemma}

\begin{proof}
We proceed by induction on $n$. For $n=1$ the statement is obvious.
Assume that the lemma is true in dimension $n-1$.
Denote by $\pi: \MN^n \to \MN^{n-1}$ the projection map
$$
  (k_1,\ldots,k_{n-1},k_n) \mapsto (k_1,\ldots,k_{n-1}).
$$
By the inductive assumption, there are finitely many vectors ${\bar v}_1,\ldots,{\bar v}_t \in \pi(R)$
such that
$$
  \pi(R) = ({\bar v}_1 +\MN^{n-1}) \cup ({\bar v}_2 +\MN^{n-1}) \cup\dots\cup ({\bar v}_t +\MN^{n-1}).
$$
Let $v_i \in R$, $i=1,\dots,t$, be any vectors such that ${\bar v}_i = \pi(v_i)$.
Obviously, if
$$
  (k_1,k_2,\ldots,k_n) \in R \setminus \bigcup_i (v_i +\MN^n)
$$
then $k_n < M_n$ where $M_n$ is the maximal $n$-th coordinate of all $v_i$.
Proceeding in a similar way for all other coordinates $i=1,2,\dots,n-1$,
we find finitely many vectors $v_1$, $v_2$, $\dots$, $v_r$ in $R$
such that every vector $(k_1,k_2,\ldots,k_n)$ in the complement
$$
  T = R \setminus \bigcup_i (v_i +\MN^n)
$$
satisfies $k_i < M_i$ for all $i=1,\dots,n$ and hence $T$ is finite. To get the required set $\set{v_i}$,
it remains to add to the set of already chosen ~$v_i$'s all vectors in $T$.
\end{proof}

\begin{proposition} \label{pr:solvable-description}
There exist finitely many functions $v_1,\ldots,v_m \in \MN^\CB$
such that
$$
  \CP = (v_1+ \CL) \cup \ldots \cup (v_m+ \CL)
$$
and therefore, $\CP$ is recursive.
\end{proposition}

\begin{proof}
Functions in $\MN^\CB$ may be viewed as vectors whose coordinates are indexed by elements of $\CB$.
For $u \in \CB$, the $u$-th coordinate of a function $\xi \in \MN^\CB$ is $\xi(u)$.
Let $\set{\la_u}_{u\in \CB}$ be the corresponding basis where, by definition, $\la_u(v) = 1$
if $u=v$ and $\la_u(v) = 0$ otherwise.
Then $\MN^\CB$ is the set of all non-negative integer linear combinations of the vectors $\la_u$.
By the definition of $\CL$, it is the set of all non-negative integer linear combinations of
vectors in a set $\set{n_u \la_u}$ for some positive integers $n_u$, $u \in \CB$.
This implies that $\MN^\CB$ can be partitioned into finitely many subsets $\tau + \CL$
(where $\tau$ runs over the corresponding ``parallelepiped'' of vectors whose coordinates ~$k_u$
satisfy $0 \le k_u < n_u$ for each $u$).

By intersecting each $\tau + \CL$ with $\CP$, we partition $\CP$ into
finitely many subsets $\tau + \CP_\tau$ with $\CP_\tau \seq \CL$.
By Lemma \ref{le:P-is-cone}, we have $\CP_\tau + \CL \seq \CP_\tau$ for each $\tau$.
Then we apply Lemma~\ref{le:convex_union} to each ~$\CP_\tau$
(writing vectors in the basis $\set{n_u \la_u}$ instead of $\set{\la_u}$).
This proves the first statement.

The second statement obviously follows from the first.
\end{proof}

\section{Boundness of the commutator width} \label{se:commutator-boundness}

In this section, we apply the technique developed in Sections \ref{se:splitting} and \ref{se:diophantine}
and prove Theorem ~\ref{thm:bounded-width}. Throughout the section, we use the notation:
$$
  R_n = [x_1,y_1] [x_2,y_2] \dots [x_n, y_n]
$$
for a standard coefficient-free orientable quadratic word of genus $n \ge 1$.

In terms of quadratic equations, the
statement of the theorem can be formulated in the following way: there is a number $N$ such that if an equation
$R_nc=1$ is solvable in $\Ga$ and $n> N$ then the equation $R_{n'}c=1$ is solvable in $\Ga$ for some $n'\le N$.
The idea of the proof (described in more detail in Section \ref{se:stability_splitting})
is to apply the splitting operation described in Section \ref{se:splitting} and to show that
it does not depend on the number of commutators in the commutator part of the equation.


\subsection{Reduced constraints on $R_n$}

The main goal of this subsection is to prove that any constraint $\gamma$ on $R_n$ modulo $K$
can be simplified and turned into some form called the {\em reduced} form.
By $\Stab(R_n)$ we denote the subgroup of all automorphisms
$\al\in\Aut(F_{\Var(R_n)})$ with $\al(R_n) = R_n$.

\begin{lemma}
For any homomorphism $\ga: F_{\Var(R_n)}\to \Z$ there exists
an automorphism $\al\in \Stab(R_n)$ such that:
\begin{gather*}
  \ga\al(x_1) = \gcd\set{\ga(x_1),\dots, \ga(x_n), \ga(y_1),\dots, \ga(y_n)}, \\
  \ga\al(x_i) = 0 \ \text{for } i \ge 2, \quad \ga\al(y_i) = 0 \ \text{for all } i=1,\dots,n.
\end{gather*}
\end{lemma}

\begin{proof}
Let $\bar F$ be the abelian quotient of $F_{\Var(R_n)}$ over the commutator subgroup.
We write elements of $\bar F$ as vectors in the basis $\set{\bar x_1,\bar y_1,\dots,\bar x_n, \bar y_n}$
where $\bar x_i$ and $\bar y_i$ are natural images of ~$x_i$ and $y_i$ in ~$\bar F$.
Any automorphism $\al\in \Aut(F_{\Var(R_n)})$ acts on $\bar F$ as an element of $\GL(2n,\Z)$.

We need to show that any vector $\bar t = (t_1,t_2,\dots,t_{2n}) \in \bar F$
can be transformed by an automorphism in $\Stab(R_n)$
to $(d, 0, \dots, 0)$ where $d = \gcd\set{t_1, t_2, \dots, t_{2n}}$.

The following automorphisms
$$
  (x_i \mapsto y_i x_i), \quad (y_i \mapsto x_i y_i)
$$
generate a subgroup of $\Stab(R_n)$ which acts on each $\Z^2$-block as $\SL(2,\Z)$.
Hence we may assume that $\bar t$ is of the form $(t_1,0,t_3,0,\dots,t_{2n-1},0)$.

The following chain
\begin{alignat*}{2}
  x_1^{-1} y_1^{-1} x_1 y_1 \cdot x_2^{-1} y_2^{-1} x_2 y_2 \quad
  && \xrightarrow{(x_1 \mapsto x_2^{-1} x_1 x_2, \ y_1 \mapsto x_2^{-1} y_1 x_2)} \quad
  & x_2^{-1} \cdot x_1^{-1} y_1^{-1} x_1 \cdot y_1 \cdot y_2^{-1} x_2 y_2 \\
  && \xrightarrow{(x_1 \mapsto x_1 x_2^{-1}, \ y_2 \mapsto y_2 y_1)} \quad
  & x_1^{-1} y_1^{-1} x_1 \cdot x_2^{-1} \cdot y_2^{-1} x_2 y_2 \cdot y_1\\
  && \xrightarrow{(x_2 \mapsto y_1 x_2 y_1^{-1}, \ y_2 \mapsto y_1 y_2 y_1^{-1})} \quad
  & x_1^{-1} y_1^{-1} x_1 y_1 \cdot x_2^{-1} y_2^{-1} x_2 y_2
\end{alignat*}
sends $(t_1, 0, t_3, 0)$ to $(t_1-t_3, 0, t_3, 0)$ and
we can permute two neighboring $\Z^2$-blocks by
$$
  (x_{i+1} \to x_{i+1}^{[x_i,y_i]}, \ y_{i+1} \to y_{i+1}^{[x_i,y_i]})
$$
This easily implies that we can act on the coordinates with odd indices
of vectors of the form $(t_1,0,t_3,0,\dots,t_{2n-1},0)$ as $\GL(n,\Z)$.
\end{proof}

\begin{remark}
The action of $\Stab(R_n)$ on $\Z^{2n}$ is equivalent to the action of extended
mapping class group $\Mod^{\pm}(S_n)$ of the closed surface $S_n$ of genus $n$ on
its homology group $H_1(S_n,\Z)$.
Then the statement of the lemma can be easily seen from the fact that $\Mod(S_n)$ acts on $H_1(S_n,\Z)$
as the symplectic group $\Sp(n,\Z)$, see for example \cite[Theorem 6.4]{Farb_Margalit}.
\end{remark}

\begin{lemma} \label{le:reducing_constraints}
Let $G$ be a polycyclic group of degree $d$.
Then for any homomorphism $\ga: F_{\Var(R_n)}\to G$,
there exists an automorphism $\al\in \Stab(R_n)$ such that
\begin{gather*}
  \al\ga(x_i) = 1 \ \text{for } i >d, \quad \al\ga(y_i) = 1 \ \text{for all } i\ge d.
\end{gather*}
\end{lemma}

\begin{proof}
We use induction on $d$.
If $G$ is cyclic then
the statement follows from the previous lemma by taking instead of ~$\ga$ any lift $F_{\Var(R_n)} \to \Z$ of
$\ga$.
Assume that $d >1$. Then $G$ has a normal polycyclic subgroup $H$ of degree $d-1$ with a cyclic quotient $G/H$.
By taking the projection $F_{\Var(R_n)} \xrightarrow\ga G \to G/H$ and using the cyclic case we find
$\al\in \Stab(R_n)$ such that $\al\ga(x_i) \in H$ for $i>2$ and $\al\ga(y_i) \in H$ for all ~$i$.
Then we apply the inductive hypothesis with $\al\ga$ instead of $\ga$ and the product
$[x_2,y_2]\dots[x_n,y_n]$ instead of ~$R_n$.
\end{proof}

By Lemma \ref{le:K_split_component}(ii),
$\Ga/K$ is the direct product of cyclic group of order $2$ generated by $bK$
and the dihedral group of order $8$ generated by $aK$ and $dK$.
Hence, $\Ga/K$ is polycyclic of degree $3$ with the subnormal series:
$$
  \Ga/K = G_0 >  G_1 > G_2 > G_3 = 1, \quad G_0/G_1 \simeq G_1/G_2 \simeq \Z/2\Z, \quad G_2 \simeq \Z/4\Z,
$$
where
$G_1 = \gp{K,b,ad}$ and $G_2 = \gp{K,ad}$.
Applying Lemma \ref{le:reducing_constraints} we immediately get

\begin{corollary}[Reducing commutator part] \label{co:reducing-constraints-K}
For any $n\ge 3$ and any homomorphism $\ga :F_{\Var(R_n)} \to \Ga/K$
there is an automorphism $\al\in\Stab(R_n)$
such that all the values $\al\ga(x_i)$ and $\al\ga(y_i)$ are trivial except, possibly, $\al\ga(x_1)$,
$\al\ga(x_2)$, $\al\ga(x_3)$, $\al\ga(y_1)$ and $\al\ga(y_2)$.
\qed
\end{corollary}

By Corollary \ref{co:reducing-constraints-K},
every constraint $\ga :F_{\Var(R_n)} \to \Ga/K$ is equivalent (with the
equivalence defined as lying in one orbit under the action of $\Stab(R_n)$) to a {\em reduced}
constraint ~$\ga'$ trivial on $\Var(R_n)$ except maybe variables $x_1,x_2,x_3,y_1,y_2$.
Reduced constraints are represented by quintuples of elements of $\Ga/K$;
for $\th = (h_1,h_2,h_3,h_4,h_5) \in (\Ga/K)^5$
by $\ga_{\th,n}$ we denote the constraint $F_{\Var(R_n)} \to \Ga/K$ defined by:
$$
\begin{array}{lll}
\ga_{\th,n}(x_i) = h_i \ \text{for}\ i=1,2,3, &\ \ & \ga_{\th,n}(x_i) = 1 \ \text{for} \ i \ge 4,\\
\ga_{\th,n}(y_i) = h_{i+3} \ \text{for}\ i=1,2, && \ga_{\th,n}(y_i) = 1 \ \text{for} \ i \ge 3.
\end{array}
$$
Fix any total order on a finite set $(\Ga/K)^5$. For $n\in\MN$ define the set of minimal
(relative to the fixed order) representatives of reduced constraints for $R_n$:
$$
\Th_n = \{\theta \in (\Ga/K)^5 \mid \forall \theta'\in (\Ga/K)^5,\ \th'\le\th,\ \ga_{\th,n}\sim\ga_{\th',n}
\Rightarrow \th'=\th \}.
$$
Clearly, $\Th_{n+1} \seq \Th_n \seq (\Ga/K)^5$ for any $n\in\MN$.
Hence, the sequence $\{\Th_i\}_{i=1}^\infty$ eventually stabilizes, i.e.,
there exists $N_0$ such that:
$$
  \Th \overset{\text{def}}= \Th_{N_0} = \Th_{N_0+1} = \Th_{N_0+2} = \dots
$$
For $\ga: F_{\Var(R_n)} \to \Ga/K$ by $\tau(\ga)$ we denote the tuple in $\Th$
representing $\ga$ up to equivalence; so we have $\ga \sim \ga_{\tau(\ga),n}$.

The effect of eventual stabilization of ascending chains of constraints
(referred below as {\em constraint saturation}) plays a key role
in the proof of Theorem \ref{thm:bounded-width}.

\subsection{Stability of splitting}
\label{se:stability_splitting}

In this subsection we describe the general proof strategy for Theorem ~\ref{thm:bounded-width}.
We consider quadratic equations of the form $R_n S =1$ where the left-hand side $R_n S$ is formally divided
into the product $R_n$ of $n$ commutators and
an orientable quadratic word $S$ with $\Var(R_n) \cap \Var(S)=1$
(so if $R_n S$ is standard then $R_n$ does not need to be all of its commutator part).
Constrained equations of this form are written as
$$
  (R_n S =1,\ga,\de)
$$
where ~$\ga$ and ~$\de$ are constraints defined on $\Var(R)$ and $\Var(S)$, respectively.
If $\ga = \ga_{\th,n}$ then the equation is {\em reduced} and we abbreviate it as
$$
  (R_n S =1,\th,\de).
$$

Every quadratic equation $R_nS=1$ in $\Ga$ is equivalent to a disjunction of reduced constrained equations:
\begin{equation}\label{eq:init_constrained_system}
\bigvee_{\substack{\theta\in\Theta,\\\delta\in\Delta}} (R_nS=1,\theta,\delta),
\end{equation}
where $\Delta$ is a set of all possible constraints on $S$.

Now let $(R_n S =1,\th,\de)$ be a standard constrained orientable quadratic equation. Applying a
splitting operation as described in Proposition \ref{pr:constrained-quadratic-splitting}
we obtain an equivalent disjunction of systems of (one or two) standard equations of the same form
$(R_{n'} S' =1,\th',\de')$. (At the moment we assume that an equation $R_{n'} S' =1$ is divided
into two parts $R_{n'}$ and $S'$ in an arbitrary way; the exact procedure will be described in
\ref{se:split_saturation}.)

Thus, applying to \eqref{eq:init_constrained_system}
a finite sequence of splittings we obtain an equivalent disjunction
of systems of quadratic equations of the form
\begin{equation}\label{eq:3}
\CQ = \bigvee_{i} \bigwedge_{j} \rb{R_{n_{i,j}} S_{i,j} = 1,\theta_{i,j},\delta_{i,j} }.
\end{equation}
Two systems of the form \eqref{eq:3},
$$
\bigvee_{i} \bigwedge_{j} \rb{R_{n_{i,j}} S_{i,j} = 1,\theta_{i,j},\delta_{i,j} }
\quad\text{and}\quad
\bigvee_{i} \bigwedge_{j} \rb{R_{k_{i,j}} S_{i,j} = 1,\theta_{i,j},\delta_{i,j} }
$$
which differ only in the genera of their commutator parts $R_{n_{i,j}}$
are called {\em similar}. For a system ~(\ref{eq:3}), define
$$
\rho(\CQ) = \min_{i,j} n_{i,j}.
$$

By $C(\CQ)$ denote the set of coefficients involved in $\CQ$. Recall that in Section ~\ref{se:diophantine}
we introduced a set $\CS$ of ``short'' elements of $\Ga$ which has the property that
after finitely many applications of splittings, the coefficients of any system \eqref{eq:3}
eventually belong to $\CS$ (see Proposition \ref{pr:coefficient-reduction}).

We will prove a fact which is formally more general than Theorem ~\ref{thm:bounded-width}.
(Theorem ~\ref{thm:bounded-width} follows if we take for $\CQ_1$ and $\CQ_2$
the systems \eqref{eq:init_constrained_system} obtained from
equations $R_N = g$ and $R_n = g$, $n > N$, where $g$ is an element of $\Ga$.)

\begin{theorem} \label{thm:equation-stability}
There exists a number $N$ with the following property.
If $\CQ_1$ and $\CQ_2$ are similar systems with $\rho(\CQ_1),\rho(\CQ_2) \ge N$
then $\CQ_1$ is solvable if and only if $\CQ_2$ is solvable.
\end{theorem}

The proof of Theorem \ref{thm:equation-stability} uses induction and consists of two major steps.

\begin{proposition}[Base of induction]\label{pr:stability_basis}
There exists a number $N_1$ such that for any two similar systems
$\CQ_1$ and $\CQ_2$ with $\rho(\CQ_1),\rho(\CQ_2) \ge N_1$ and $C(\CQ_i)\subseteq \CS$,
$\CQ_1$ is solvable if and only if $\CQ_2$ is solvable.
\end{proposition}

\begin{proposition}[Stability of splitting]\label{pr:stability_split}
There exists a number $N_2$ such that application of the splitting operation to
similar systems $\CQ_1$ and $\CQ_2$ with $\rho(\CQ_1),\rho(\CQ_2) \ge N_2$
results in similar systems $\CQ_1'$ and $\CQ_2'$ with $\rho(\CQ_i') \ge \rho(\CQ_i)$.
\end{proposition}

Let us check that Propositions \ref{pr:stability_basis} and \ref{pr:stability_split}
imply Theorem \ref{thm:equation-stability}. Take $N = \max(N_1,N_2)$.
Let $\CQ_1$ and $\CQ_2$ be two similar systems of the form  \eqref{eq:3} with $\rho(\CQ_i) \ge N$.
By Proposition ~\ref{pr:stability_split} splitting of $\CQ_1$ and $\CQ_2$
results in similar systems $\CQ_1'$ and $\CQ_2'$. Each $\CQ_i'$ is equivalent to $\CQ_i$ and
since $\rho(\CQ_i') \ge N$, we are again
under conditions of Proposition \ref{pr:stability_split}.
Continuing the splitting process
we eventually obtain two similar systems with coefficients in $\CS$
(by Proposition \ref{pr:coefficient-reduction}).
Then by Proposition \ref{pr:stability_basis} one is solvable if and only if the other is solvable.
Q.E.D.

We prove Propositions \ref{pr:stability_basis} and \ref{pr:stability_split}
in subsections \ref{se:split_basis} and \ref{se:split_saturation}, respectively.

\subsection{Base of induction}
\label{se:split_basis}

For the proof of Proposition \ref{pr:stability_basis}, it is enough to consider the case of a single equation:

\begin{lemma} \label{le:reduction-base}
There is a number $N_1$ with the following property.
Assume that $n,n' \ge N_1$ and all coefficients of $S$ have length at most $3$. Then
the equation $(R_n S=1, \th, \de)$ is solvable if and only if the equation
$(R_{n'} S=1, \th, \de)$ is solvable.
\end{lemma}

\begin{proof}
The equation $(R_{n'} S=1, \ga_{\th,n'}, \de)$ is obtained from $(R_n S=1, \ga_{\th,n}, \de)$ by
inserting a word
$$
  W = [x_{n+1}, y_{n+1}] \dots [x_{n'}, y_{n'}]
$$
and extending the constraint by setting
$\ga_{\th,n'}(x_i) = \ga_{\th,n'}(y_i) = 1$ for all $x_i,y_i \in \Var(W)$.

Let $(Q=1,\ze)$ be an ordered form of the equation $(R_n S=1, \ga, \de)$
(see Step 4 in Section ~\ref{se:diophantine}).
As described in the proof of Lemma \ref{le:ordering_terms}, to get this form we
apply automorphisms to $R_n S$ to re-order the commutator and the coefficient parts.
To get an ordered form of $(R_{n'} S=1, \ga_{\th,n'}, \de)$ we can use automorphisms
\begin{align*}
  U x W V &\xrightarrow{(x_i \mapsto x^{-1} x_i x, \ y_i \mapsto x^{-1} y_i x, \ i=n+1,\dots,n')}
  U W x V, \\
  U W x V &\xrightarrow{(x_i \mapsto x x_i x^{-1}, \ y_i \mapsto y_i x^{-1}, \ i=n+1,\dots,n')}
  U x W V
\end{align*}
which can move $W$ at any position in $R_{n'}S$ without changing the constraint on the variables
$x_i,y_i \in \Var(W)$. This easily implies that an ordered form of the equation $(R_{n'} S=1, \ga', \de)$
can be written as $(Q'=1, \ze')$ where $Q'$ is obtained from $Q$ by inserting $W$ at an appropriate
position in $Q$ and extending $\ze$ by defining
$\ze'(x_i) = \ze'(y_i) = 1$ for $x_i,y_i \in \Var(W)$.

Let $\la_{Q,\ze}$ and $\la_{Q',\ze'}$ be corresponding codes defined in Step 5, Section \ref{se:diophantine}.
We see immediately that $\la_{Q',\ze'}$ and $\la_{Q,\ze}$ differ in a single coordinate by $m$, i.e.\
$$
  \la_{Q',\ze'} = \la_{Q,\ze} + m \mu
$$
where $\mu$ is is defined by $\mu((1,1)) = 1$ on $(1,1) \in \Ga/K \times \Ga/K$ and $\mu(u) = 0$ for all other $u
\in \CB$.
Now Proposition \ref{pr:solvable-description} implies that
there exist positive numbers $N_1$ and $M$ such that if $n \ge N_1$ and $m$ is a multiple of $M$ then the
solvability of $(Q=1,\ze)$ is equivalent to the solvability of $(Q'=1,\ze')$.
Since the solvability of $(Q'=1,\ze')$ implies the solvability of the same equation with $n'$ changed to any $n''$
with $n < n'' < n'$ (we can substitute $x_i = y_i = 1$ for any extra commutator $[x_i,y_i]$)
we can drop the condition that $m$ is a multiple of $M$.

Finally, we observe that $N_1$ can be chosen independently on the choice of the equation $(R_n S=1, \ga, \de)$
(we can take $N_1$ as the maximal coordinate of all vectors $v_i$ in Proposition ~\ref{pr:solvable-description}.)
\end{proof}

\subsection{Constraint saturation}
\label{se:split_saturation}

Here we prove Proposition \ref{pr:stability_split}.
It is enough to consider the case when $\CQ_1$ and $\CQ_2$ consist of a single equation.

Fix an arbitrary $S$, a constraint $\delta$ for $S$, a tuple
$\theta \in \Th$ and consider an equation
$$
\CQ^{(n)} = \left (R_3\prod_{i=1}^n[x_i,y_i] \cdot S = 1,\ \theta,\ \delta \right).
$$
Splitting this equation (without subsequent reduction to the standard form)
we obtain an equivalent disjunction
$$
  \CQ^{(n)}_1 = \bigvee_{\substack{\lambda\in\Lambda,\\\pi_1,\ldots,\pi_n,\\\delta'\in\Delta}}
\left( \left\{
\begin{array}{l}
Q_0 \prod_{i=1}^n[x_i,y_i] S_0 = 1, \\[0.5ex]
Q_1 \prod_{i=1}^n[x_i',y_i'] S_1 = 1,
\end{array}
\right.
\ \lambda,\pi_1,\ldots,\pi_n,\delta' \right),
$$
where:
\begin{itemize}
\item
$\Psi(R_3) = (Q_0,Q_1)$ and
$\lambda$ are constraints on $\Var(Q_0)\cup \Var(Q_1)$;
\item
each $\pi_i$ is a constraint on $\{x_i,y_i,x_i',y_i'\}$;
\item
$\Psi(S) = (S_0,S_1)$ and
$\delta'$ are constraints on $\Var(S_0)\cup \Var(S_1)$;
\item
$\Lambda$ and $\Delta$ are sets of constraints which do not depend on $n$;
\item
up to renaming variables,
each $\pi_i$ runs over a fixed set $\Pi$ of constraints on $\set{x,y,x',y'}$.
\end{itemize}

\subsubsection*{Saturation in the disjoint case}
If the two equations in $\CQ^{(n)}_1$ have disjoint sets of variables then
both are in the standard form.
In this case, reducing the set of constraints on $Q_0 \prod_{i=1}^n[x_i,y_i]$
and on $Q_1 \prod_{i=1}^n[x_i',y_i']$ we obtain a new system
$$
  \CQ^{(n)}_2 = \bigvee_{\substack{(\theta_0,\theta_1)\in\Phi_n(\La) ,\\ \delta'\in\Delta}}
\left( \left\{
\begin{array}{l}
Q_0 \prod_{i=1}^n[x_i,y_i] \cdot S_0=1, \\[0.5ex]
Q_1 \prod_{i=1}^n[x_i',y_i'] \cdot S_1=1,
\end{array}
\right.
\ \theta_0,\theta_1,\delta' \right),
$$
where each $\theta_i$ define a constraint on $\Var(Q_i)$,
$\Psi_n(\La) \subseteq \Theta^2$ and all variables $\{x_i,y_i,x_i',y_i'\}$ are trivially constrained.
The set $\Pi$ contains, in particular, the trivial constraint on $\{x,y,x',y'\}$.
This implies $\Psi_n(\La)\subseteq\Psi_{n+1}(\La)$. Since there are finitely many possible choices of ~$\La$,
starting from some $n\ge N_2'$ we get
$\Psi_n(\La)=\Psi_{n+1}(\La)$ for any $n$.
Then systems $\CQ^{(n)}_2$ are similar for different values of $n\ge N_2'$
and thus Proposition \ref{pr:stability_split} holds in this case.

\subsubsection*{Saturation in the non-disjoint case}
If the two equations in $\CQ^{(n)}_1$  have a shared variable,
we need to compute
\begin{equation}\label{eq:2}
\left( Q_0 \prod_{i=1}^n[x_i,y_i] S_0 \right) \: \# \: \left( Q_1 \prod_{i=1}^n[x_i',y_i'] S_1 \right)
\end{equation}
and then take it to the standard form. Up to interchanging the two commutator subsequences,
(\ref{eq:2}) is of the form
$$
U \, \prod_{i=1}^n[x_i,y_i] \: V \, \prod_{i=1}^n[x_i',y_i'] \: W.
$$
Applying $(x_i \mapsto U^{-1}x_i U, \ y_i \mapsto U^{-1}y_i U,
\ x_i' \mapsto (UV)^{-1}x_i' UV, \ y_i' \mapsto (UV)^{-1}y_i' UV)$
we obtain a word
$$
\prod_{i=1}^n[x_i,y_i] \, \prod_{i=1}^n[x_i',y_i'] \: U V W,
$$
which is the same as
$$
\prod_{i=1}^n[x_i,y_i] \, \prod_{i=1}^n[x_i',y_i'] \cdot Q_0S_0 \# Q_1S_1.
$$
Thus, $\CQ^{(n)}_1$  is equivalent to the disjunction
$$
\bigvee_{\substack{\pi_1,\ldots,\pi_n\in\Pi',\\ \lambda\in\Lambda,\\ \delta'\in\Delta}}
  \rb{\prod_{i=1}^n[x_i,y_i] \, \prod_{i=1}^n[x_i',y_i'] \cdot Q_0S_0 \# Q_1S_1 = 1, \
  \pi_1,\ldots,\pi_n,\lambda,\delta'},
$$
where $\Pi'$ is a set of constraints on $\set{x,y,x',y'}$ (and inclusions $\pi_i \in \Pi'$ are assumed up
to renaming variables). Note that $\Pi'$ contains the trivial constraint on $R_2$
since it is obtained from $\Pi$ by an appropriate conjugation of values of variables.

After reduction to the standard form, we obtain a disjunction
$$
  \CQ^{(n)}_3 = \bigvee_{\substack{\ \th' \in \Psi_n,\\ \xi \in \Xi}}
  \rb{R_{2n} S' = 1, \
  \th',\xi},
$$
where $S'$ is the standard form of $Q_0S_0 \# Q_1S_1$, $\th'$ is a constraint on $\Var(R_{2n})$ and
$\xi$ is a constraint on $\Var(S')$ (we do not change constraints on $R_{2n}$ by Remark \ref{rm:no-conjugation}).
The sequence $\set{\Psi_n}$
is ascending and since there are finitely many possible choices of such sequences (determined by the possible
choices
of $\Pi'$), for some $N_2''$ we have stabilization: $\Psi_n = \Psi_{n+1}$ for all $n \ge N_2''$
and any starting equation $\CQ^{(n)}_0$.
Then, again, systems $\CQ^{(n)}_3$ are similar for different values of ~$n$.

Proposition \ref{pr:stability_split} is proved for $N_2 = \max(3,N_2',N_2'')$.
This finishes the proof of Theorem \ref{thm:equation-stability}.

\bibliography{main_bibliography}

\end{document}